\documentclass[11pt]{article}
\usepackage{url}
\usepackage{enumerate}
\usepackage{amsmath}%
\usepackage{stmaryrd}
\usepackage{amsxtra}%
\usepackage{amsfonts}%
\usepackage{amssymb}%
\usepackage{amsthm}
\usepackage{graphicx}
\usepackage{subfigure}

\usepackage[margin=3cm]{geometry}

\newcommand{\gam}[1]{\gamma^{}_{#1}\hspace{-0.75mm}}
\newcommand{\gamw}[1]{\gamma^{}_{#1}\hspace{-3mm}'\hspace{2mm}}
\newcommand{\gami}{\gamma^{}_{i}\hspace{-0.25mm}}
\newcommand{\limsups}[1]{\underset{#1}{\operatorname{lim}\operatorname{sup}\ensuremath{^*}}\;}
\newcommand{\liminfs}[1]{\underset{#1}{\operatorname{lim}\operatorname{inf}\ensuremath{_*}}\;}
\newcommand{\iid}{\emph{i.i.d.}}

\newcommand{\almostsurely}{{\rm \ \ \ a.s.}}
\newcommand{\K}{{\mathcal K}}
\renewcommand{\H}{{\mathcal H}}
\newcommand{\A}{{\mathcal A}}
\newcommand{\B}{{\mathcal B}}
\renewcommand{\O}{{\mathcal O}}
\renewcommand{\bar}[1]{{\overline{#1}}}
\newcommand{\F}{{\mathcal F}}
\newcommand{\X}{{\mathcal X}}
\renewcommand{\b}{\vb{b}}
\newcommand{\R}{\mathbb{R}}
\newcommand{\N}{\mathbb{N}}
\newcommand{\Z}{\mathbb{Z}}
\newcommand{\supp}{{\rm supp}}
\newcommand{\eps}{\varepsilon}
\newcommand{\vb}[1]{\mathbf{#1}}
\renewcommand{\tilde}[1]{\widetilde{#1}}
\renewcommand{\phi}{\varphi}
\newcommand*\samethanks[1][\value{footnote}]{\footnotemark[#1]}
\newtheorem{theorem}{Theorem}
\theoremstyle{plain}

\newtheorem{definition}{Definition}

\newtheorem{lemma}{Lemma}

\newtheorem{proposition}{Proposition}
\newtheorem{remark}{Remark}

\numberwithin{equation}{section}

\title{A Hamilton-Jacobi equation for the continuum limit of non-dominated sorting\thanks{The research described in this paper was partially supported by ARO grant W911NF-09-1-0310 and NSF grants CCF-1217880 and DMS-0914567.}}

\author{Jeff Calder\thanks{Department of Mathematics, University of Michigan. ({\tt \{jcalder,esedoglu\}@umich.edu})}
   \and Selim Esedo\=glu\samethanks
   \and Alfred O. Hero\thanks{Department of Electrical Engineering and Computer Science, University of Michigan. ({\tt hero@eecs.umich.edu})}}

\begin{document}

\maketitle

\begin{abstract}
We show that non-dominated sorting of a sequence $X_1,\dots,X_n$ of \iid~random variables in $\R^d$ has a continuum limit that corresponds to solving a Hamilton-Jacobi equation involving the probability density function $f$ of $X_i$.  Non-dominated sorting is a fundamental problem in multi-objective optimization, and is equivalent to finding the canonical antichain partition and to problems involving the longest chain among Euclidean points.  As an application of this result, we show that non-dominated sorting is asymptotically stable under bounded random perturbations in $X_1,\dots,X_n$.  We give a numerical scheme for computing the viscosity solution of this Hamilton-Jacobi equation and present some numerical simulations for various density functions.  
\end{abstract}

\section{Introduction}
\label{sec:intro}

Let $X_1,\dots,X_n$ be \iid~random variables on $\R^d$ with density function $f\in L^1(\R^d)$.  The points form a partially ordered set $\X_n=\{X_1,\dots,X_n\}$ under the partial order 
\begin{equation}\label{eq:porder}
x \leqq y  \iff x_i \leq y_i \  \text{ for } \ i=1,\dots,d.
\end{equation}
Let $\ell(n)$ denote the length of a longest chain\footnote{A \emph{chain} is a totally ordered subset of $\X_n$.} in $\X_n$, and for $x \in \R^d$ let $u_n(x)$ denote the length of a longest chain in $\X_n$ consisting of points less than $x$.  We are interested in the asymptotic properties of $u_n$ as $n\to \infty$. 

When $f$ is a smooth density on $[0,1]^d$, hence $\ell(n)=u_n(1,\dots,1)$, the problem of studying the asymptotics of $\ell(n)$ has a long history.  It begins with Ulam's famous problem~\cite{ulam1961} of finding the length of a longest increasing subsequence of a random permutation.  Hammersley~\cite{hammersley1972} made some of the first breakthroughs in understanding Ulam's problem.  He observed that the distribution of the length of a longest increasing subsequence among $n$ numbers chosen uniformly at random is the same as the distribution of $\ell(n)$ for uniformly distributed points on $[0,1]^2$.  Using subadditive ergodic theory, Hammersley showed that $n^{-\frac{1}{2}} \ell(n)$ converges almost surely to a constant $c$ as $n\to \infty$, and he conjectured that $c=2$.  In subsequent papers, Vershik and Kerov~\cite{vershik1977} and Logan and Shepp~\cite{logan1977} showed that $c\leq2$ and $c\geq 2$, respectively.  Hammersley's results were generalized by Bollob\'as and Winkler~\cite{bollobas1988} to uniformly distributed points on $[0,1]^d$;  they showed that there exist positive constants $c_d$ such that $n^{-\frac{1}{d}}\ell(n) \to c_d$ almost surely as $n \to \infty$, and $c_d \nnearrow e$ as $d\to \infty$.  The only known values of $c_d$ are $c_1=1$ and $c_2=2$.  Deuschel and Zeitouni~\cite{deuschel1995} generalized Hammersley's results in another direction.  For $X_1,\dots,X_n$ \iid~on $[0,1]^2$ with $C^1$ density function $f:[0,1]^2 \to \R$, bounded away from zero, they showed that $n^{-\frac{1}{2}}\ell(n) \to 2\bar{J}$ in probability, where $\bar{J}$ is the supremum of the energy
\[J(\phi) = \int_0^1 \sqrt{\phi'(x) f(x,\phi(x))} \, dx,\]
over all $\phi:[0,1]\to[0,1]$ nondecreasing and right continuous.  

There is another motivation for studying the asymptotics of $u_n$ that arises in multi-objective optimization problems.  Such problems are of immense importance in many fields of science and engineering, including control theory and path planning~\cite{mitchell2003,kumar2010,madavan2002}, gene selection and ranking~\cite{speed2003,hero2002,hero2003,hero2004,fleury2002,fleury2004a,fleury2004b}, data clustering~\cite{handl2005}, database systems~\cite{kossmann2002,papadias2005} and image processing and computer vision~\cite{mumford1989,chan2001}.  In a discrete multi-objective optimization problem, one has several objective functions $g_i: S \to [0,\infty)$, where $i=1,\dots,d$ and $S=\{x^1,\dots,x^n\}$ is a finite set, and is tasked with finding an element $x\in S$ that minimizes \emph{all} of the functions simultaneously.  This is generally an impossible task, and instead, a family of solutions are obtained based on the notion of Pareto-optimality.  A feasible solution $x \in S$ is called \emph{Pareto-optimal} if for every $y \in S$, we have $g_i(y) > g_i(x)$ for some $i$, or $g_i(y)=g_i(x)$ for all $i$; in other words, no other feasible solution is better in every objective.  The collection of Pareto-optimal elements is denoted $\F_1$ and called the \emph{first Pareto front}.  It is the most natural notion of solution for a discrete multi-objective optimization problem. If we set $X_i=(g_1(x^i),\dots,g_d(x^i)) \in \R^d$ for $i=1,\dots,n$, then assuming all $X_i$ are distinct, it is not hard to see that 
\[x^i \in \F_1 \iff u_n(X_i)=1.\]
The second Pareto front, $\F_2$, consists of the Pareto-optimal elements of $S\setminus \F_1$, and in general
\[\F_k = {\rm Pareto \  optimal \  elements \ of \ } S \setminus \bigcup_{j< k} \F_j.\]
The Pareto front that a particular feasible solution lies on is useful for ranking feasible solutions. 
As before, when the $X_i$ are all distinct we have 
\[x^i \in \F_k \iff u_n(X_i) = k.\]
This observation is essential.  It says that studying the asymptotic shapes of the Pareto fronts $\F_1,\F_2,\dots$ is equivalent to studying the longest chain function $u_n$.  
Figure \ref{fig:example-fronts} shows the Pareto fronts for $n=50$ points uniformly distributed on $[0,1]^2$, and Figure \ref{fig:example-fronts-10^6} shows the Pareto fronts for $n=10^6$ points.    The points $X_i$ that are on the same Pareto front are connected by a continuous staircase curve that represents the jump set of $u_n$.  

\begin{figure}
\centering
\subfigure[$n=50$ points]{\includegraphics[clip=true, trim = 50 30 30 20, width=0.45\textwidth]{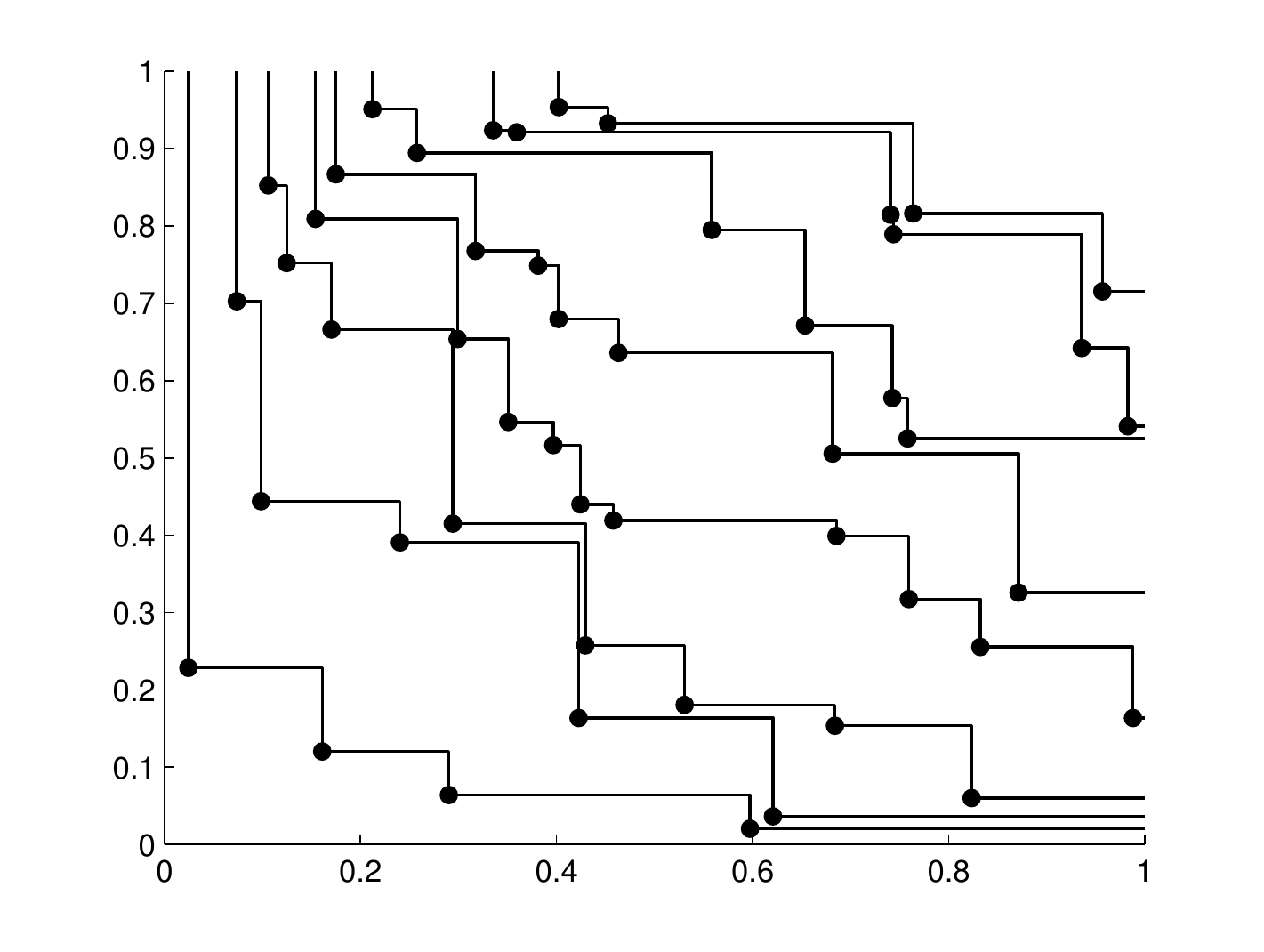}\label{fig:example-fronts}}
\subfigure[$n=10^6$ points]{\includegraphics[clip = true, trim = 55 35 30 20, width=0.45\textwidth]{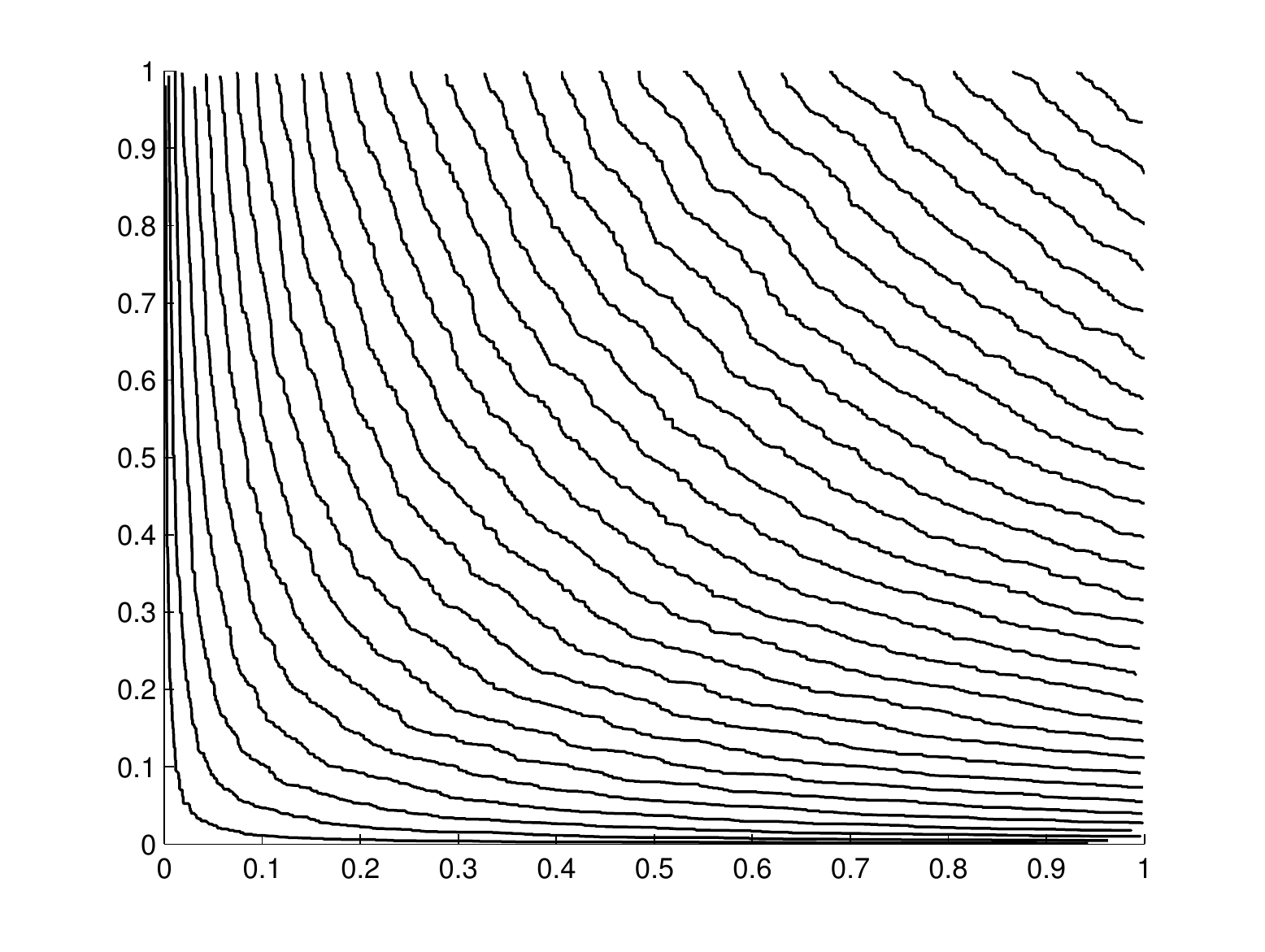}\label{fig:example-fronts-10^6}}
\caption{Examples of Pareto fronts for $X_1,\dots X_n$ chosen from the uniform distribution on $[0,1]^2$.  In (b), $29$ equally spaced fronts are depicted.}
\end{figure}
In the multi-objective optimization literature, the process of computing the Pareto fronts for a collection of points is called \emph{non-dominated sorting}~\cite{deb2002}.  In the combinatorics literature, the partition $S = \F_1 \cup \F_2 \cup \cdots$ is called the \emph{canonical antichain partition}~\cite{felsner1999}.  Although we have described non-dominated sorting in the context of a discrete optimization problem, it is a fundamental tool in continuous optimization as well.  Many state of the art algorithms for continuous optimization involve a large number of discrete subproblems, each of which requires non-dominated sorting.  The most common examples are the so-called genetic and evolutionary algorithms for continuous multi-objective optimization~\cite{deb2002,fonseca1993,fonseca1995,deb2001,srinivas1994}.  The applications of non-dominated sorting are not restricted to optimization; indeed, there are further striking applications in combinatorics~\cite{felsner1999,lou1993}, molecular biology~\cite{pevzner2000,adhar2007}, graph theory~\cite{lou1993}, Young Tableaux~\cite{viennot1984,felsner1999} and even in physical layout problems in the design of integrated circuits~\cite{adhar2007}.

The goal of this paper is to study the asymptotics of $u_n$, and hence the asymptotics of non-dominated sorting.  Our main result, Theorem \ref{thm:linfty-conv}, states that $n^{-\frac{1}{d}}u_n$ converges almost surely to a continuous function $U$, which is the viscosity solution of a Hamilton-Jacobi equation.   Our proof is based on linking the asymptotics of $u_n$ to a variational problem, which is a generalization of the variational problem discovered by Deuschel and Zeitouni~\cite{deuschel1995} to higher dimensions.  The Hamilton-Jacobi equation satisfied by $U$ is the Hamilton-Jacobi-Bellman equation~\cite{bardi1997} for the corresponding variational problem.   We describe our main result in Section \ref{sec:main}, and postpone the proofs to Sections \ref{sec:var} and \ref{sec:proof}.  In Section \ref{sec:num}, we give a numerical scheme for computing $U$, and show simulation results comparing the level sets of $U$ to Pareto fronts for various density functions.

\subsection{Main result}
\label{sec:main}

For $x,y \in \R^d$, we write $x \leq y$ if $x\leqq y$ and $x\neq y$.  When $x_i < y_i$ for $i=1,\dots,d$, we write $x<y$, and we set $\R^d_+ = \{x \in \R^d \, : \, x>0\}$.  We will always assume $d\geq 2$.  For $s,t \in \R$, $s\leq t$ and $s< t$ will retain their usual definitions.  Let $\Omega \subset \R^d$ and let $f: \R^d\to [0,\infty)$.  We place the following assumptions on $f$ and $\Omega$.
\begin{itemize}
\item[(H1)] There exists a continuous nondecreasing function $m: [0,\infty) \to [0,\infty)$ satisfying $m(0)=0$ such that
\[|f(x)-f(y)| \leq m(|x-y|),\]
for $x,y \in \Omega$, and $f(x) = 0$ for $x \not\in \Omega$,
\item[(H2)] $\Omega \subset \R^d_+$ is open and bounded with Lipschitz boundary.
\end{itemize}
Set
\begin{equation}\label{eq:Adef}
\A = \left\{\gamma \in C^1([0,1]; \R^d) \, : \, {\gamma \: }'(t) \geq 0 \ \text{ for all } \ t \in [0,1]\right\}.
\end{equation}
Recall that ${\gamma \: }'(t)\geq 0$ means that $\gami'(t) \geq 0$ for $i=1,\dots,d$ and ${\gamma \: }'(t)\neq 0$.
Define $J:\A \to [0,\infty)$ by
\begin{equation}\label{eq:Jdef}
J(\gamma) = \int_0^1 f(\gamma(t))^\frac{1}{d} (\gam{1}'(t)\cdots \gam{d}'(t))^\frac{1}{d} \, dt,
\end{equation}
and $U:\R^d\to \R$ by
\begin{equation}\label{eq:vardef}
U(x) = \sup_{\gamma \in \A \, : \, \gamma(1) \leqq x} J(\gamma).
\end{equation}
We make the following definition.
\begin{definition}\label{def:pm}
Given a domain $\O\subset \R^d$, we say that a function $u:\O \to \R$ is \emph{Pareto-monotone} if $x\leqq y \implies u(x) \leq u(y)$ for all $x,y \in \O$.  
\end{definition}

In Section \ref{sec:var}, we show that $U$ is a Pareto-monotone viscosity solution of the Hamilton-Jacobi partial differential equation (PDE)
\begin{equation}\label{eq:hjb-main}
\begin{cases}
U_{x_1} \cdots U_{x_d} = \frac{1}{d^d} f& \ \ \text{ on } \ \  \R^d_+  \\
U = 0& \ \ \text{ on } \ \  \partial \R^d_+.
\end{cases}
\end{equation}
The PDE \eqref{eq:hjb-main} should be interpreted as the Hamilton-Jacobi-Bellman equation for the value function $U$.  
We note that $f$ need only be Borel-measurable, bounded and have compact support in $\R^d_+$ for $U$ to be a viscosity solution of \eqref{eq:hjb-main}.  The stronger assumptions (H1) and (H2) are needed to prove that $U$ is the unique Pareto-monotone viscosity solution of \eqref{eq:hjb-main} (see Theorem \ref{thm:uniq}) under an appropriate boundary condition at infinity.
 Our main result is
\begin{theorem}\label{thm:linfty-conv}
Let $f$ satisfy (H1), let $\Omega$ satisfy (H2), and let $X_1,\dots,X_n$ be \iid~with density $f$.
Then there exists a positive constant $c_d$ such that 
\[n^{-\frac{1}{d}} u_n \longrightarrow c_d U \ \ \text{ in } \  L^\infty(\R^d) \  \text{ almost surely.}\]
\end{theorem}
The constants $c_d$ are the same as those given by Bollob\'as and Winkler~\cite{bollobas1988}.  In particular, $c_1=1$, $c_2=2$ and $c_d\nnearrow e$ as $d\to\infty$.   When $f$ is a product density, i.e., $f(x)=f_1(x_1) \cdots f_d(x_d)$,  the value function $U$ is given by
\begin{equation}\label{eq:prod-den-main}
U(x) = \left(\int_{0 \leqq y \leqq x} f(y) \, dy\right)^\frac{1}{d} =\left(\int_{0}^{x_i} f_1(t) \, dt\right)^\frac{1}{d} \cdots \left(\int_{0}^{x_d} f_d(t) \, dt\right)^\frac{1}{d}.
\end{equation}

For the case $f=1$ and $d=2$, Aldous and Diaconis~\cite[p.~204]{aldous1995} provided a non-rigorous derivation of \eqref{eq:hjb-main} by viewing the problem as an interacting particle process. They used this to motivate their proof that $c=2$ in Ulam's problem, but make no rigorous statements about the relationship between \eqref{eq:hjb-main} and the longest chain problem.  A similar, though tangentially related, PDE also appears in growth models in multiple dimensions that are defined through the height of a random partial order~\cite[p.~209]{seppalainen2007}.  

Theorem \ref{thm:linfty-conv} provides a new tool with which to study the asymptotics of non-dominated sorting and the longest chain problem.  As an example of the applicability of this result, we show in Theorem \ref{thm:stability} that non-dominated sorting is asymptotically stable under bounded random perturbations.   
Evidently, Theorem \ref{thm:linfty-conv} reduces the problem of non-dominated sorting to solving a Hamilton-Jacobi equation.  From an algorithmic perspective, this may be useful in designing fast approximate algorithms for non-dominated sorting, or finding lengths of longest chains.  We study some of these applications in a subsequent paper~\cite{calder2013b}.


\subsection{Motivation}
\label{sec:heur}

As motivation, let us give an informal derivation of the Hamilton-Jacobi PDE \eqref{eq:hjb-main}.
Suppose $f:\R^d\to\R$ is continuous and $n^{-\frac{1}{d}}u_n \to U \in C^1(\R^d)$ uniformly. Fix $n$ large enough so that $n^{-\frac{1}{d}}u_n \approx U$.  Then the $k^{\rm th}$ Pareto front should be well approximated by the level set $\{y \, : \, U(y) = n^{-\frac{1}{d}}k\}$.   It is not hard to see that $U$ should be Pareto-monotone (recall Definition \ref{def:pm}), and hence it is reasonable to assume that $U_{x_i} > 0$ for all $i$.  Fix $x,v \in \R^d$ with $\langle DU(x), v\rangle > 0$, where $DU(x)$ denotes the gradient of $U$ at $x$, and consider the quantity $n^{\frac{1}{d}}(U(x+v) -U(x))$.  This is approximately the number of Pareto fronts passing between $x$ and $x+v$.  When counting these fronts, we may restrict ourselves to the region
\[A =\{y \, : \, U(y) \geq U(x) \ \text{ and } \ y \leqq x+v\}.\]
This is because any samples in $\{y \, : \, U(y) < U(x)\}$ will be on a previous Pareto front and only samples that are less than $x+v$ can influence the Pareto rank of $x+v$. See Figure \ref{fig:derivation} for a depiction of this region and some quantities from the derivation.
\begin{figure}
\centering
\fbox{\includegraphics[width=0.6\textwidth]{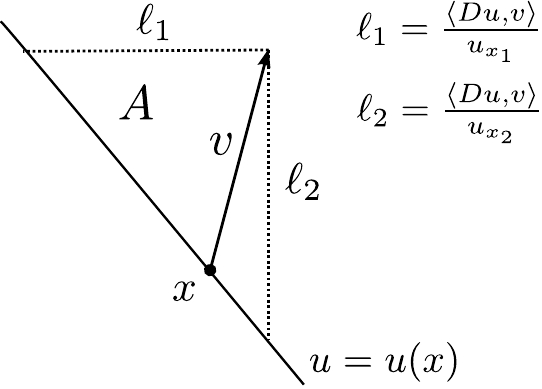}}
\caption{Some quantities from the informal derivation of the Hamilton-Jacobi PDE \eqref{eq:hjb-main}.}
\label{fig:derivation}
\end{figure}
Since $U_{x_i}(x) > 0$ for all $i$, and $U$ is $C^1$, $A$ is well approximated by a simplex for small $|v|$, and furthermore, the samples within $A$ are approximately uniformly distributed.  Let $m$ denote the number of samples falling in $A$.   By scaling the simplex into a standard simplex, without disrupting the Pareto ordering within $A$, it is reasonable to conjecture that the number of Pareto fronts within $A$ (or the length of a longest chain in $A$) is approximately $cm^{\frac{1}{d}}$ for some constant $c$, independent of $x$.  For simplicity we take $c=1$.

By the law of large numbers, we have $m \approx n\int_A f(y) \, dy$.  Hence when $|v|>0$ is small we have
\begin{equation}\label{eq:usim}
n^\frac{1}{d}(U(x+v)-U(x)) \approx \left(n\int_A f(y) \, dy\right)^{\frac{1}{d}} \approx n^{\frac{1}{d}}|A|^{\frac{1}{d}}f(x)^\frac{1}{d},
\end{equation}
where $|A|$ denotes the Lebesgue measure of $A$.
Let $\ell_1,\dots,\ell_d$ denote the side lengths of the simplex $A$.  Then $|A|\approx c \, \ell_1\cdots \ell_d$ for a constant $c$ which we again take to be $1$.  Since $x+v-\ell_i e_i$ lies approximately on the tangent plane to the level set $\{y \, : \, U(y)=U(x)\}$, we see that
\[\langle DU(x), v - \ell_i e_i\rangle \approx 0.\]
Rearranging the above we see that $\ell_i \approx U_{x_i}(x)^{-1}\langle D U(x),v\rangle$, and hence
\begin{equation}\label{eq:measureA}
|A|\approx U_{x_1}(x)^{-1} \cdots U_{x_d}(x)^{-1} \langle DU(x),v\rangle^d.
\end{equation}
For small $|v|$, we can combine \eqref{eq:measureA} and \eqref{eq:usim} to obtain
\[\langle DU(x), v\rangle \approx U(x+v) - U(x) \approx f(x)^\frac{1}{d}U_{x_1}(x)^{-\frac{1}{d}} \cdots U_{x_d}(x)^{-\frac{1}{d}} \langle DU(x),v\rangle.\]
Simplifying, we see that $U$ should satisfy
\begin{equation}\label{eq:hjb-der}
U_{x_1} \cdots U_{x_d} =f  \ \ \text{ on } \ \  \R^d,
\end{equation}
up to scaling by a constant.

Although this derivation is informal, it is straightforward and conveys the essence of the result.  It is difficult, however, to construct a rigorous proof based on these heuristics.  There are two main reasons for this.  First, it supposes that $n^{-\frac{1}{d}}u_n$ converges to a limit $U$, which is not obvious. Second, it is essential that $U \in C^1$, as we require $A$ to be an approximate simplex.  Solutions of \eqref{eq:hjb-der} are in general not smooth, and can have points of non-differentiability due to crossing characteristics.  This is true even in the case that $f$ is smooth, and is related to the geometry of $\Omega$.

\section{Analysis of variational problem}
\label{sec:var}

Before studying the variational problem \eqref{eq:vardef}, we recall some aspects of the theory of optimal control~\cite{bardi1997} that are relevant to our problem.  We will describe the infinite horizon optimal control problem, but the discussion below applies with minor modifications to other variants of optimal control, such as finite horizon or undiscounted problems with exit times.  The state of the control problem, $y(t)$, is assumed to obey the dynamics
\begin{equation}\label{eq:dynamics}
\begin{cases}
y'(t) = g(y(t),\alpha(t)),& t>0 \\
y(0)=x,
\end{cases}
\end{equation}
where $\alpha:[0,\infty) \to A$ is the control, $A$ is a topological space, and $g:\R^d \times A \to \R^d$.  
Given an initial condition $x \in \R^d$, the solution of \eqref{eq:dynamics} is denoted $y_x(\cdot)$.  Let 
\[\A:=\{{\rm measurable \ functions \ } [0,\infty) \to A\}.\]
The goal in optimal control is to select the control $\alpha \in \A$ to minimize the cost functional
\begin{equation}\label{eq:cost}
J(x,\alpha) : = \int_0^\infty c(y_x(t),\alpha(t)) e^{-\lambda t} \, dt,
\end{equation}
where $\lambda > 0$ and $c:\R^d \times A \to \R$. The value function for this problem is
\begin{equation}\label{eq:value}
v(x) := \inf_{\alpha \in \A} J(x,\alpha).
\end{equation}
Under sufficient regularity assumptions on $c$ and $g$ (discussed below), the value function is a H\"older- (or Lipschitz) continuous viscosity solution of the Hamilton-Jacobi-Bellman equation
\begin{equation}\label{eq:hjb-ex}
\lambda v + H(x,Dv)  = 0 \ \ \text{ on } \ \  \R^d,
\end{equation}
where
\begin{equation}\label{eq:hamiltonian}
H(x,p) =  \sup_{a \in A} \{ -\langle g(x,a), p\rangle - c(x,a) \}.
\end{equation}

Although the variational problem \eqref{eq:vardef} can be cast in this framework, the assumptions on the running cost $c(\cdot,\cdot)$ in the existing literature are too restrictive.  For our variational problem, we have $\lambda=0$, $g(x,a)=a$, $A=\R^d_+$,
\begin{equation}\label{eq:our-cost}
c(x,a) = -f(x)^\frac{1}{d} (a_1\cdots a_d)^\frac{1}{d},
\end{equation}
and $U(x) = -v(x)$.
In the proofs of Theorems \ref{thm:pointwise-upper} and \ref{thm:pointwise-lower}, we require the standard optimal control theory to hold for $f$ piecewise constant on arbitrarily small grids.  In the standard reference on optimal control~\cite{bardi1997}, it is assumed that $x \mapsto c(x,a)$ is uniformly continuous.  This assumption is then used to prove regularity of the value function $v$. There is relatively little research devoted to relaxing the regularity condition on $c$.  There are some results for the optimal control problem associated with the Eikonal equation~\cite{newcomb1995,camilli2006,deckelnick2004}, which allow $c$ to have discontinuities.  These results assume that $A=\R^d$ and make essential use of either Lipschitzness of $v$, or uniform continuity and/or coercivity of $p\mapsto H(x,p)$, none of which hold for the variational problem \eqref{eq:vardef}.  Soravia~\cite{soravia2002} and Garavello and Soravia~\cite{garavello2004} considered a running cost of the form $c(x,a) = c_1(x,a) + c_2(x)$, where $c_1$ is continuous and $c_2$ is Borel-measurable, and showed that the standard optimal control results hold with minor modifications. This is incompatible with \eqref{eq:our-cost} when $f$ is not continuous.  A similar program is carried out for differential games here~\cite{garavello2006}.  Barles et al.~\cite{barles2011} study optimal control on multi-domains, where the discontinuity in $c$ is assumed to lie in a half-space.

Under the assumption that $f$ is compactly supported, bounded and  Borel-measurable, the standard results on optimal control hold for the variational problem \eqref{eq:vardef} with minor modifications to the proofs.  In particular, in Lemma \ref{lem:holder} we show that $U$ is H\"older-continuous with exponent $\frac{1}{d}$, and in Theorem \ref{thm:hjb} we show that $U$ is a viscosity solution of \eqref{eq:hjb-main}.  The uniqueness of viscosity solutions of \eqref{eq:hjb-main} under the assumption that $f$ satisfies (H1) is a more delicate problem.  This is addressed in Section \ref{sec:comp}. 

In our main result, Theorem \ref{thm:linfty-conv}, we assume that $f$ satisfies (H1), which is stronger than Borel-measurability. We assume Borel-measurability in much of this section so that our results apply to piecewise constant densities, which are used to approximate $f$ in the proofs of Theorems \ref{thm:pointwise-upper} and \ref{thm:pointwise-lower}.  To be more precise,  we set
\begin{equation}\label{eq:function-space}
\B = \{f:\R^d \to \R \, : \, f \text{ is bounded, Borel-measurable, and }\supp(f) \subset [0,1]^d\}.
\end{equation}
We note that the assumption $\supp(f) \subset [0,1]^d$ is not restrictive, as we can make a simple scaling argument to obtain the case where $f$ has compact support in $\R^d$.  
We also note that Borel-measurability of $f$ is necessary, as opposed to Lebesgue-measurability, to guarantee that the composition $t \mapsto f(\gamma(t))$ is Lebesgue measurable.  

We now introduce some new notation.  We will write $\gamma \leqq x$ whenever $\gamma(t) \leqq x$ for all $t\in [0,1]$.  We write $\gamma_1 \leqq \gamma_2$ whenever $\gamma_1(1) \leqq \gamma_2(0)$.  The same definitions apply to $\leq,<, \geqq,\geqq$ and $>$ with obvious modifications.  
For $y\in \R^d$ and $r>0$ we set $B_r(y) = \{x \in \R^d \, : \, |x-y| < r\}$.
For $x,y \in \R^d$ we set
\begin{equation}
w(x,y) = \begin{cases}
\sup \{ J(\gamma) \, : \, \gamma \in \A  \ \text{ and } \ x \leqq \gamma \leqq y\}& {\rm if} \ x \leqq y \\
0& {\rm otherwise}.\end{cases}
\end{equation}

\subsection{Basic properties of $U$}
\label{sec:basic}

We establish here some basic properties of $U$.  Namely, in Lemma \ref{lem:holder} we establish H\"older-continuity of $U$, and in Lemma \ref{lem:dpp}, we establish a dynamic programming principle for $U$.
\begin{lemma}\label{lem:holder}
Let $f \in \B$.  Then $U$ is H\"older-continuous with exponent $\frac{1}{d}$ and H\"older seminorm $[U]_{\frac{1}{d}}\leq \|f\|^\frac{1}{d}_{L^\infty(\R^d)}$.
\end{lemma}
\begin{proof}
Let $x,z \in \R^d$ and let $\eps >0$.   Choose $\gamma \in \A$ with $\gamma \leqq x$ and $J(\gamma) \geq U(x) - \eps$. Since $f(x) = 0$ for $x \not\in [0,1]^d$, we may assume that $\gamma(t) \in [0,1]^d$ for all $t\in[0,1]$. Set
\[s = \sup \{ t \in [0,1] \, : \, \gamma(t) \leqq z\}.\]
If for all $t\in [0,1]$ we have $\gamma(t) \not\leqq z$, then set $s=0$.  We claim that
\begin{equation}\label{eq:holder-a}
U(z) \geq U(x)- \int_s^1 f(\gamma(t))^\frac{1}{d} (\gam{1}'(t) \cdots \gam{d}'(t))^\frac{1}{d} \, dt  - \eps.
\end{equation}
To see this: In the case that $s >0$, we have $\gamma(s) \leqq z$ and hence
\begin{align*}
U(z) &{}\geq{} \int_0^s f(\gamma(t))^\frac{1}{d} (\gam{1}'(t) \cdots \gam{d}'(t))^\frac{1}{d} \, dt \notag \\
&{}={} J(\gamma) - \int_s^1 f(\gamma(t))^\frac{1}{d} (\gam{1}'(t) \cdots \gam{d}'(t))^\frac{1}{d} \, dt\notag \\
&{}\geq{} U(x)- \int_s^1 f(\gamma(t))^\frac{1}{d} (\gam{1}'(t) \cdots \gam{d}'(t))^\frac{1}{d} \, dt  - \eps.
\end{align*}
In the case that $s=0$, we have
\begin{align*}
U(z) \geq 0 &{}={} J(\gamma) - \int_0^1 f(\gamma(t))^\frac{1}{d} (\gam{1}'(t) \cdots \gam{d}'(t))^\frac{1}{d} \, dt\notag \\
&{}\geq{} U(x)- \int_0^1 f(\gamma(t))^\frac{1}{d} (\gam{1}'(t) \cdots \gam{d}'(t))^\frac{1}{d} \, dt  - \eps.
\end{align*}
Hence  \eqref{eq:holder-a} is established.  Suppose $s<1$.  Then there must exist $i$ such that $\gamma_i(s) \geq z_i$.  It follows that
\[\int_s^1 \gami'(t)\, dt =  \gamma_i(1) - \gamma_i(s) \leq x_i - z_i = |x_i-z_i|. \]  
Applying the generalized H\"older inequality we see that
\begin{align*}
\int_s^1 f(\gamma(t))^\frac{1}{d} (\gam{1}'(t) \cdots \gam{d}'(t))^\frac{1}{d} \, dt &{}\leq{} \|f\|^\frac{1}{d}_{L^\infty(\R^d)}\left( \int_s^1 \gam{1}'(t) \, dt\right)^\frac{1}{d} \cdots \left(\int_s^1 \gam{d}'(t) \, dt \right)^\frac{1}{d} \\
&{}\leq{} \|f\|^\frac{1}{d}_{L^\infty(\R^d)} |x_i-z_i|^\frac{1}{d} \prod_{j\neq i} (\gamma_j(1)-\gamma_j(s))^\frac{1}{d} \\
&{}\leq{} \|f\|^\frac{1}{d}_{L^\infty(\R^d)} |x_i-z_i|^\frac{1}{d}.
\end{align*}
Inserting this into \ref{eq:holder-a} we obtain
\begin{equation}\label{eq:holder-b}
U(x) - U(z) \leq \|f\|^\frac{1}{d}_{L^\infty(\R^d)} |x_i-z_i|^\frac{1}{d} + \eps \leq \|f\|^\frac{1}{d}_{L^\infty(\R^d)} |x-z|^\frac{1}{d} + \eps 
\end{equation}
If $s=1$ then inspecting \eqref{eq:holder-a}, we see that $U(x) - U(z) \leq \eps$, which implies \eqref{eq:holder-b}.  Sending $\eps \to 0$ we find that $U(x) - U(z) \leq \|f\|^\frac{1}{d}_{L^\infty(\R^d)} |x-z|^\frac{1}{d}$.  We can reverse the roles of $x$ and $z$ in the preceding argument to obtain the opposite inequality. 
\end{proof}
\begin{remark}\label{rem:wremark}
By a similar argument, we can show that $w:\R^d\times \R^d \to \R$ is H\"older-continuous with exponent $1/d$ and $[w]_{\frac{1}{d}} \leq \|f\|^\frac{1}{d}_{L^\infty(\R^d)}$.
\end{remark}
\begin{lemma}[Dynamic Programming Principle]\label{lem:dpp}
Let $f \in \B$.  Then for any $r>0$ and $y \in \R^d$ we have
\begin{equation}\label{eq:dpp}
U(y) = \max_{x \in \partial B_r(y) \, : \, x \leq y} \{U(x) + w(x,y)\}.
\end{equation}
\end{lemma}
\begin{proof}
Let us denote the right hand side of \eqref{eq:dpp} by $v(y)$.  We first show that $U(y) \leq  v(y)$. Let $\eps>0$ and let $\gamma \in \A$ such that $\gamma \leqq y$ and $J(\gamma) \geq U(y) - \eps$. Suppose that $|\gamma(1)-y| \geq r$.  Then there exists $x \in \partial B_r(y)$ such that $\gamma(1) \leqq x\leqq y$ and hence 
\[U(y)\leq J(\gamma) + \eps \leq U(x) + \eps \leq v(y) + \eps.\]
If $|\gamma(0) - y| \leq r$ then there exists $x \in \partial B_r(y)$ such that $x \leq \gamma \leq y$ and hence 
\[v(y) \geq w(x,y) \geq J(\gamma) \geq U(y) - \eps.\]
Finally, suppose that $|\gamma(1) - y| < r$ and $|\gamma(0) - y| > r$.  Then there exists $0<s < 1$ such that $|\gamma(s)-y|=r$.  Set $x=\gamma(s)$ and define $\gamma^1,\gamma^2 \in \A$ by
\[\gamma^1(t) = \gamma(st) \ \text{ and } \ \gamma^2(t) = \gamma(s + t(1-s)) \ \text{ for } \ t \in [0,1].\]
Note that $\gamma^1 \leqq x$ and $x \leqq \gamma^2 \leqq y$.
Since $J$ is invariant under a change of parametrization of $\gamma$, we see that
\[U(y) \leq J(\gamma)+\eps = J(\gamma^1) + J(\gamma^2)+\eps \leq U(x) + w(x,y)+\eps \leq v(y) + \eps.\]
Sending $\eps\to 0$ we obtain $U(y) \leq v(y)$.

We now show that $U(y) \geq v(y)$.  By Lemma \ref{lem:holder} and Remark \ref{rem:wremark}, there exists $x \in \partial B_r(y)$ with $x \leq y$ such that
\[v(y) = U(x) + w(x,y).\]
Let $\eps>0$ and let $\gamma^1,\gamma^2 \in \A$ with $\gamma^1 \leqq x$ and $x \leqq \gamma^2 \leqq y$ such that
\[J(\gamma^1) \geq U(x) - \frac{\eps}{2} \ \text{ and } \ J(\gamma^2)\geq w(x,y) - \frac{\eps}{2}.\]
Since $\gamma^1\leqq \gamma^2\leqq y$, we can concatenate $\gamma^1$ and $\gamma^2$ to find that $U(y)\geq J(\gamma^1) + J(\gamma^2)$.  Thus we have
\[U(y) \geq U(x) + w(x,y) - \eps = v(y) - \eps.\]
Sending $\eps\to 0$ yields $U(y) \geq v(y)$.
\end{proof}

\subsection{Hamilton-Jacobi-Bellman equation for $U$}
\label{sec:viscosity}

We digress momentarily to recall the definition of viscosity solution of
\begin{equation}\label{eq:genHJ}
H(Du) = f  \ \ \text{ on } \ \   \O,
\end{equation}
where $\O \subset \R^d$ is open, $H:\R^d \to \R$ is continuous, $f:\O \to \R$ is bounded, and $u: \O \to \R$ is the unknown function.  For more information on viscosity solutions of Hamilton-Jacobi equations,  we refer the reader to~\cite{bardi1997}.  The \emph{superdifferential} of $u$ at $x \in \O$, denoted $D^+u(x)$, is the set of all $p \in \R^d$ satisfying
\begin{equation}\label{eq:superjet}
u(y) \leq u(x) + \langle p,y-x\rangle + o(|x-y|) \ \text{ as } \ \O \ni y \to x.
\end{equation}
Similarly, the \emph{subdifferential} of $u$ at $x\in \O$, denoted $D^-u(x)$, is the set of all $p \in \R^d$ satisfying
\begin{equation}\label{eq:subjet}
u(y) \geq u(x) + \langle p,y-x\rangle + o(|x-y|) \ \text{ as } \ \O \ni y \to x.
\end{equation}
Equivalently, we may set
\[D^+ u(x) = \{ D\phi(x) \, : \, \phi \in C^1(\O) \ \text{ and } \ u-\phi \ {\rm has \ a \ local \ max \ at \ } x\},\]
and
\[D^- u(x) = \{ D\phi(x) \, : \, \phi \in C^1(\O) \ \text{ and } \ u-\phi \ {\rm has \ a \ local \ min \ at \ } x\}.\]
\begin{definition}
A \emph{viscosity subsolution} of \eqref{eq:genHJ} is a continuous function $u:\O \to \R$ satisfying
\begin{equation}
H(p) \leq f^*(x) \ \text{ for all } \ x \in \O \ \text{ and } \ p \in D^+u(x).
\end{equation}
Similarly, a \emph{viscosity supersolution} of \eqref{eq:genHJ} is a continuous function $u: \O \to \R$ satisfying
\begin{equation}
H(p) \geq f_*(x) \ \text{ for all } \ x \in \O \ \text{ and } \ p \in D^-u(x).
\end{equation}
\end{definition}

The functions $f_*$ and $f^*$ are the lower and upper semicontinuous envelopes of $f$, respectively, defined by
\[f^*(x) = \limsup_{r\ssearrow 0} \ \{ f(y) \, : \, y \in \O \ \text{ and } \  |x-y| \leq r \},\]
and $f_* = -(-f)^*$.
If $u$ is a viscosity subsolution and supersolution of \eqref{eq:genHJ}, then we say that $u$ is a \emph{viscosity solution} of \eqref{eq:genHJ}.  

After a basic proposition, we establish in Theorem \ref{thm:hjb} that $U$ is a \emph{Pareto-monotone} viscosity solution of \eqref{eq:hjb-main}. 
\begin{proposition}\label{prop:monotone}
Let $\O \subset \R^d$ be open and let $v:\O \to \R$ be continuous and Pareto-monotone. Then 
\[D^+ v(x)\cup D^- v(x) \subset \bar{\R^d_+} \ \text{ for all } \ x \in \O.\]
\end{proposition}
\begin{proof}
Let $x \in \O$ and $p \in D^+ v(x)$. For any index $i$ and small enough $t>0$, we have $x \leqq x + te_i \in \O$.  Since $v$ is Pareto-monotone, we have
\[v(x) \leq v(x+te_i) \leq v(x) + p_i t + o(t) \ \text{ as } \ t\ssearrow 0.\]
Hence $p_i \geq o(t)/t$ as $t\ssearrow 0$ which implies that $p_i \geq 0$. The proof for $D^- v(x)$ is similar.
\end{proof}
\begin{theorem}\label{thm:hjb}
Let $f \in \B$.  Then the value function $U$ defined by \eqref{eq:vardef} is a Pareto-monotone viscosity solution of the Hamilton-Jacobi equation
\begin{equation}\label{eq:hjb}
U_{x_1} \cdots U_{x_d} = \frac{1}{d^d} f \ \ \text{ on } \ \  \R^d.
\end{equation}
Furthermore, $U$ satisfies
\begin{itemize}
\item[(i)] Whenever $\supp(f) \subset \{x \in \R^d \, : \, 0 \leqq x \leqq z\}$, we have 
\[U(x_1,\dots,x_d) = U(\min(x_1,z_1),\cdots \min(x_d,z_d)) \ {\rm for \ all } \ x \in \R^d_+,\]
\item[(ii)] $U(x) = 0$ for every $x \in \R^d \setminus \R^d_+$.
\end{itemize}
\end{theorem}
\begin{proof}
It follows from the definition of $U$ \eqref{eq:vardef} that $U$ is Pareto-monotone, and (ii) follows from the fact that $\supp(f) \subset [0,1]^d$.

For (i), let $z \in \R^d$ such that $\supp(f) \subset \{x \in \R^d \, : \, 0 \leqq x\leqq z\}$ and let $x \in \R^d_+$ such that $x_i > z_i$ for some $i$.  Set $\hat{x} = (\min(x_1,z_1),\dots,\min(x_d,z_d))$. Since $U$ is Pareto-monotone we have $U(\hat{x}) \leq U(x)$.  Let $\eps > 0$ and $\gamma \in \A$ such that $\gamma \leqq x$ and $U(x) \leq J(\gamma) + \eps$.  Let
\[s = \sup \{ t \, : \, \gamma(t) \leqq \hat{x}\}.\]
If for all $t \in [0,1]$ we have $\gamma(t) \not\leqq \hat{x}$, then set $s=0$.
If $s = 1$, then $\gamma \leqq \hat{x}$ and hence $U(x) \leq J(\gamma)+\eps \leq U(\hat{x}) + \eps$.  If $s=0$ then for every $t \in [0,1]$, $\gamma(t) \not\in \supp(f)$, and hence $J(\gamma)=0$.   It follows that 
\[U(x) \leq J(\gamma) + \eps=\eps \leq U(\hat{x}) + \eps.\]
If $0<s < 1$, then for any $t>s$, $\gamma_i(t) > z_i$ for some $i$, and hence $f(\gamma(t)) = 0$.  Set $\gamma^1(t) = \gamma(st)$ for $t\in [0,1]$.   Then $\gamma^1 \leqq \hat{x}$ and $J(\gamma)=J(\gamma^1)$, hence $U(x)  \leq J(\gamma) + \eps = J(\gamma^1) + \eps \leq U(\hat{x}) + \eps$.  Sending $\eps \to 0$ we see that $U(x) \leq U(\hat{x})$ and hence $U(x) = U(\hat{x})$.  

We now show that $U$ is a viscosity supersolution of \eqref{eq:hjb}.  
Let $y \in \R^d$, let $a \in \R^d_+$, and set $\gamma(t) = y - a(1-t)$.   By Lemma \ref{lem:dpp} we have
\begin{align}\label{eq:hjb-c}
U(y) &{}\geq{} U(y - a(1-t)) + \int_t^1 f(y - a(1-s))^\frac{1}{d} (a_1\cdots a_d)^\frac{1}{d} \, ds \notag\\
&{}\geq{} U(y - a (1-t)) + (1-t)f_{*}(y)(a_1\cdots a_d)^\frac{1}{d} + o(1-t) \ \text{ as } \ t \nnearrow 1.
\end{align}
Let $p \in D^-U(y)$.  Since $y - a(1-t) \to y$ as $t \nnearrow 1$, we have
\begin{align*}
\langle p, (1-t) a\rangle &{}\stackrel{\eqref{eq:subjet}}{\geq} U(y) - U(y - a(1-t)) + o(1-t) \\
&{}\stackrel{\eqref{eq:hjb-c}}{\geq}(1-t)f_{*}(y) (a_1 \cdots a_d)^\frac{1}{d} + o(1-t) \ \text{ as } \ t \nnearrow 1.
\end{align*}
Sending $t \nnearrow 1$ we obtain
\[\langle p,a\rangle \geq f_{*}(y)^\frac{1}{d} (a_1\cdots a_d)^\frac{1}{d}.\]
Since $a> 0$ was arbitrary, we obtain
\begin{equation}\label{eq:hjb-a}
\sup_{a> 0} \left\{ -\langle p,a\rangle + f_{*}(y)^\frac{1}{d}(a_1\cdots a_d)^\frac{1}{d}\right\} \leq 0.
\end{equation}
Since $U$ is Pareto-monotone, Proposition \ref{prop:monotone} yields $p\geq 0$.  Hence if $f_{*}(y)=0$ then $U$ is trivially a viscosity supersolution of \eqref{eq:hjb} at $y$. We may therefore suppose that $f_{*}(y) > 0$.  
Fix $i$ and set $a_j=1$ for $j\neq i$.  By \eqref{eq:hjb-a} we have 
\[\sup_{a_i > 0}\left\{ -\sum_{j\neq i} p_j +  a_i^\frac{1}{d}f_{*}(y)^\frac{1}{d} - a_ip_i\right\} \leq 0.\]
Since $a_i$ can be arbitrarily large, we must have $p_i> 0$ for the above to hold.  Substituting $a_i = p_i^{-1}$ into \eqref{eq:hjb-a} and simplifying we obtain
\[p_1 \cdots p_d \geq \frac{1}{d^d}f_{*}(y).\]
Thus $U$ is a viscosity supersolution of \eqref{eq:hjb}. 

We now show that $U$ is a viscosity subsolution of \eqref{eq:hjb}.
Let $y \in \R^d$, let $\eps>0$, and let $p \in D^+U(y)$.   By Lemmas \ref{lem:holder} and \ref{lem:dpp} and Remark \ref{rem:wremark}, for every $r>0$ there exists $x\in \partial B_r(y)$ with $x \leq y$ such that $U(y) = U(x) + w(x,y)$. Hence there exists $\gamma \in \A$ with $x \leqq \gamma \leqq y$ such that
\[U(y) \leq \int_0^1 f(\gamma(t))^\frac{1}{d} (\gam{1}'(t) \cdots \gam{d}'(t))^\frac{1}{d} \, dt + U(x) + \eps r.\]
By H\"older's inequality
\begin{align*}
U(y) - U(x)&{}\leq{} \int_0^1 f(\gamma(t))^\frac{1}{d} (\gam{1}'(t) \cdots \gam{d}'(t) )^\frac{1}{d} \, dt + \eps r \\
&{}\leq{} (f^{*}(y)^\frac{1}{d} + o(1))\left(\int_0^1 \gam{1}'(t) \, dt \right)^\frac{1}{d} \cdots\left(\int_0^1 \gam{d}'(t) \, dt \right)^\frac{1}{d}  + \eps r\\
&{}\leq{} f^{*}(y)^\frac{1}{d}|x_1-y_1|^\frac{1}{d} \cdots |x_d-y_d|^\frac{1}{d} +  o(r) + \eps r,
\end{align*}
as $r\ssearrow 0$.  Since $x \to y$ as $r \ssearrow 0$, we have
\[\langle p,y-x\rangle \stackrel{\eqref{eq:superjet}}{\leq} U(y) - U(x) + o(r) \leq f^{*}(y)^\frac{1}{d}|x_1-y_1|^\frac{1}{d} \cdots |x_d-y_d|^\frac{1}{d} + o(r) + \eps r,\]
as $r \ssearrow 0$.
Choose $r>0$ small enough so that $o(r)/r \leq \eps$, and set $a=(y-x)/r$.  Then we have
\[ -\langle p,a \rangle + f^{*}(y)^\frac{1}{d}(a_1\cdots a_d)^\frac{1}{d} \geq  -2\eps.\]
Since $\eps>0$ was arbitrary, we see that
\[\sup_{a \geq 0 \, : \, |a|=1} \left\{-\langle p,a \rangle + f^{*}(y)^\frac{1}{d}(a_1\cdots a_d)^\frac{1}{d}\right\} \geq  0.\]
Since $U$ is Pareto-monotone, we have $p\geq 0$. If $p_i =0$ for some $i$, then $p_1\cdots p_d \leq f^{*}(y)/d^d$.  Thus we may assume that $p_i > 0$ for all $i$.  Then the supremum above is attained at some $a > 0$ with $|a|=1$.  By scaling $a$ so that $a_1\cdots a_d =1$, we see that
\begin{equation}\label{eq:hjb-b}
\sup_{a>0 \, : \, a_1 \cdots a_d=1} \left\{-\langle p,a \rangle + f^{*}(y)^\frac{1}{d}\right\} \geq  0.
\end{equation}
Since $p_i > 0$ for all $i$, we have that
\[\limsup_{|a|\to \infty, \ a >0} -\langle p,a\rangle + f^*(y)^\frac{1}{d} = -\infty.\]
It follows that the supremum in \eqref{eq:hjb-b} is attained at some $a^* > 0$.  
Introducing a Lagrange multiplier $\lambda > 0$, the necessary conditions for $a^*$ to be a maximizer of the above constrained optimization problem are
\[p_i = \frac{\lambda}{a^*_i} {\rm \ \ for \ all \ }i \in \{1,\dots,d\} \ \text{ and } \ a^*_1\cdots a^*_d = 1.\]
It follows that $\lambda  = (p_1\cdots p_d)^\frac{1}{d}$ and $a^*_i = p_i^{-1} (p_1\cdots p_d)^\frac{1}{d}$.
Substituting this into \eqref{eq:hjb-b} we have
\[p_1 \cdots p_d \leq \frac{1}{d^d} f^{*}(y),\]
which completes the proof.
\end{proof}
\begin{remark}\label{rem:truncation}
We remark that $U$ satisfies an important truncation property.  Namely, if we fix $z \in \R^d$ and define $\widetilde{U},\widehat{f}:\R^d \to \R$ by
\[\widetilde{U}(x_1,\dots,x_d) = U(\min(x_1,z_1),\dots,\min(x_d,z_d))\]
\[\widehat{f}(x) = \begin{cases} f^*(x) & \text{if } x \leqq z \\ 0 & {\rm otherwise,}\end{cases}\]
then we have
\[\widetilde{U}_{x_1} \cdots \widetilde{U}_{x_d} = \widehat{f} \ \ \text{ on } \ \  \R^d\]
in the viscosity sense.
Indeed, this follows directly from Theorem \ref{thm:hjb} by noting that $\supp(\widehat{f}) \subset \{x \in \R^d\, : \, 0 \leqq x\leqq z\}$ and
\[\widetilde{U}(x) = \sup_{\gamma \in \A \, : \, \gamma \leqq x} \int_0^1 \widehat{f}(\gamma(t))^\frac{1}{d} (\gam{1}'(t) \cdots \gam{d}'(t))^\frac{1}{d} \, dt.\]
\end{remark}

\subsection{Comparison principle}
\label{sec:comp}

We aim here to establish that $U$ is the unique viscosity solution of \eqref{eq:hjb-main} under hypotheses (H1) and (H2) on $f$ and $\Omega$, which in general allow $f$ to be discontinuous.  The standard results on uniqueness of viscosity solutions~\cite{bardi1997,crandall1992} assume uniformly continuous dependence on spatial variables. There has been some recent work relaxing this condition, as it is important in many applications.  Tourin~\cite{tourin1992} considered Hamilton-Jacobi equations of the form $H(x,Du)=0$, where $x \mapsto H(x,p)$ is allowed to have a discontinuity along a smooth surface, and proved a comparison principle under the assumption that $p\mapsto H(x,p)$ is convex and uniformly continuous.  Neither assumption holds for \eqref{eq:hjb-main}, although the non-convexity can be easily remedied.  Deckelnick and Elliot~\cite{deckelnick2004} prove a comparison principle for Lipschitz viscosity solutions of Eikonal-type equations of the form $H(Du) = f$, where $f$ satisfies a regularity condition similar to (H1), but slightly more general.  As exhibited by the solution $U(x)=(x_1\cdots x_d)^\frac{1}{d}$ of \eqref{eq:hjb-main} for $f=1$, solutions of \eqref{eq:hjb-main} are not in general Lipschitz continuous.  Camilli and Siconolfi~\cite{camilli2003} proposed a new notion of viscosity solution for Hamilton-Jacobi equations in which $H$ has measurable dependence on the spatial variable $x$.  They obtain general uniqueness results under the assumption that $p \mapsto H(x,p)$ is quasiconvex and coercive.  Their results do not apply to \eqref{eq:hjb-main} due to the coercivity assumption.

The main result in this section, Theorem \ref{thm:uniq}, establishes uniqueness for \eqref{eq:hjb-main} under hypotheses (H1) and (H2), and a boundary condition at infinity.  Let us give a sketch of the proof now.  Let $u$ be a Pareto-monotone viscosity solution of \eqref{eq:hjb-main}.  We first prove a standard comparison principle, in Theorem \ref{thm:comparison}, for uniformly continuous $f$.  We can then define the regularized value functions $U_\eps$ and $U^\eps$ by replacing $f$ by its inf and sup convolutions $f_\eps$ and $f^\eps$, respectively, in \eqref{eq:vardef}.  Since $f^\eps$ and $f_\eps$ are Lipschitz continuous, the comparison principle from Theorem \ref{thm:comparison} yields $U_\eps \leq u \leq U^\eps$.  The proof is completed by showing that $U_\eps,U^\eps \to U$ as $\eps \to 0$, where $U$ is the value function defined by \eqref{eq:vardef}.  We establish a more general result in Lemma \ref{lem:l1-pert}, the proof of which relies on the second comparison principle, Theorem \ref{thm:comparison2}.  This comparison principle holds for $f$ and $\Omega$ satisfying (H1) and (H2) under the additional assumption that the subsolution is truncatable, as per Definition \ref{def:trunc}.  As pointed out in Remark \ref{rem:truncation}, the value function $U$ is truncatable, so Theorem \ref{thm:comparison2} is applicable in the proof of Lemma \ref{lem:l1-pert}.

\begin{theorem}\label{thm:comparison}
Suppose $f:\R^d_+ \to \R$ is uniformly continuous with $\supp(f) \subset [0,1]^d$.  Let $u$ and $v$ be viscosity sub- and supersolutions, respectively, of 
\begin{equation}\label{eq:pde-gen}
u_{x_1}\cdots u_{x_d} = f \ \ \text{ on } \ \  \R^d_+,
\end{equation}
and suppose that 
\begin{equation}\label{eq:projection}
u(x) = u(\min(x_1,1),\dots,\min(x_d,1)) \ \text{ for } \ x \in \R^d_+,
\end{equation}
and $v$ is Pareto-monotone.  If $u\leq v$ on $\partial \R^d_+$ then $u\leq v$ on $\R^d_+$.   
\end{theorem}

The proof of Theorem \ref{thm:comparison} utilizes the method of doubling the variables, which is standard in the theory of viscosity solutions~\cite{bardi1997}, with appropriate modifications for the boundary condition \eqref{eq:projection}.
\begin{proof}
For $\theta >0$, set $v_\theta(x) = v(x) + \theta^\frac{1}{d}\langle x, \vb{1}_d\rangle$, where $\vb{1}_d = (1,\dots,1)\in \R^d$.  Fix $x \in \R^d_+$ and $p \in D^-v_\theta(x)$. It is easy to see that $p- \theta^\frac{1}{d} \vb{1}_d \in D^-v (x)$. Hence we have
\[(p_1 - \theta^\frac{1}{d})\cdots (p_d-\theta^\frac{1}{d}) - f(x)  \geq 0.\]
Since $v$ is Pareto-monotone, we have $p_i - \theta^\frac{1}{d} \geq 0$ for all $i$, and therefore
\begin{align*}
p_1\cdots p_d = (p_1 - \theta^\frac{1}{d} + \theta^\frac{1}{d}) \cdots (p_d - \theta^\frac{1}{d} + \theta^\frac{1}{d}) &\geq (p_1-\theta^\frac{1}{d}) \cdots (p_d-\theta^\frac{1}{d}) + \theta \\
&\geq f(x) + \theta.
\end{align*}
Hence $v_\theta$ is a viscosity supersolution of
\[v_{\theta,x_1} \cdots v_{\theta,x_d} = f + \theta \ \ \text{ on } \ \  \R^d_+,\]
and $u \leq v_\theta$ on $\partial \R^d_+$.

Suppose that $\delta:=\sup_{\R^d_+} (u - v_\theta) >0$.  For $x,y \in \R^d_+$ and $\alpha > 0$, set
\[\Phi_\alpha(x,y) = u(x)-v_{\theta}(y) - \alpha|x-y|^2,\]
and $M_\alpha= \sup_{\R^d_+ \times \R^d_+}\Phi_\alpha$.  Setting $x=y$ in $\Phi_\alpha(x,y)$ we see that $M_\alpha \geq \delta  > 0$
for all $\alpha$.  Let 
\[\hat{x} = (\min(x_1,1),\dots,\min(x_d,1)) \ \text{ and } \ \hat{y} = (\min(y_1,1),\cdots,\min(y_d,1)).\]
Note that $|\hat{x} - \hat{y}| \leq |x-y|$, $u(\hat{x}) = u(x)$ (by \eqref{eq:projection}) and $v_\theta(\hat{y}) \leq v_\theta(y)$ for all $x,y\in \R^d_+$.  It follows that $\Phi_\alpha(\hat{x},\hat{y}) \geq \Phi_\alpha(x,y)$, and hence $\Phi_\alpha$ attains a maximum at some $(x_\alpha,y_\alpha) \in [0,1]^d\times [0,1]^d$.  Since $M_\alpha \geq \delta$, we have
\begin{equation}\label{eq:uv}
u(x_\alpha) - v_\theta(y_\alpha) \geq \delta + \alpha|x_\alpha - y_\alpha|^2.
\end{equation}
It follows that 
\begin{equation}\label{eq:xyalpha}
|x_\alpha - y_\alpha| \leq \frac{1}{\sqrt{\alpha}} \left( \|u\|_{L^\infty((0,1)^d)} - v_\theta(0)\right)^\frac{1}{2}.  
\end{equation}
Since $u \leq v_\theta$ on $\partial \R^d_+$ and $(x,y) \mapsto u(x) - v_\theta(y)$ is continuous, it follows from \eqref{eq:uv} and \eqref{eq:xyalpha} that $(x_\alpha,y_\alpha) \in (0,1]^d \times (0,1]^d$ for $\alpha$ large enough. For such $\alpha$ we have $p:=2\alpha(x_\alpha - y_\alpha) \in  D^+ u(x_\alpha)\cap D^-v_{\theta}(y_\alpha)$, and hence
\[p_1\cdots p_d \leq  f(x_\alpha) \  \text{ and }  \ p_1\cdots p_d \geq f(y_\alpha) + \theta.\]
Subtracting the above equations yields $ f(x_\alpha) - f(y_\alpha)\geq \theta > 0$, which contradicts the uniform continuity of $f$ and \eqref{eq:xyalpha} as $\alpha \to \infty$.  Therefore $u \leq v_\theta$ on $\R^d_+$.  Sending $\theta \to 0$ completes the proof.
\end{proof}

We can prove a comparison principle for discontinuous $f$ by assuming that the subsolution satisfies the truncation property described in Remark \ref{rem:truncation}.
For this, we make the following definition.
\begin{definition}\label{def:trunc}
Let $u$ be a viscosity subsolution of 
\begin{equation}\label{eq:trunc-eq}
u_{x_1} \cdots u_{x_d} = f \ \ \text{ on } \ \  \R^d_+.
\end{equation}
We say that $u$ is \emph{truncatable} if for every $z \in \R^d_+$, $\tilde{u}$ is a viscosity subsolution of
\[\tilde{u}_{x_1} \cdots \tilde{u}_{x_d} = \widehat{f} \ \ \text{ on } \ \  \R^d_+,\]
where $\tilde{u}$ and $\widehat{f}$ are defined in Remark \ref{rem:truncation}.
\end{definition}

We note that that truncatability is well-defined, i.e., it depends only on $f^*$.  
By Remark \ref{rem:truncation}, the value function $U$ is truncatable.  It is easy to see that every $C^1$ Pareto-monotone subsolution of \eqref{eq:trunc-eq} is truncatable.  It turns out, thanks to Theorem \ref{thm:uniq}, that every Pareto-monotone viscosity solution of \eqref{eq:trunc-eq} satisfying \eqref{eq:projection} is truncatable.  

We place the following assumptions on $f:\R^d_+ \to [0,\infty)$ and $\Omega$.
\begin{itemize}
\item[(H1)] There exists a continuous nondecreasing function $m: [0,\infty) \to [0,\infty)$ satisfying $m(0)=0$ such that
\[|f(x)-f(y)| \leq m(|x-y|),\]
for $x,y \in \Omega$, and $f(x) = 0$ for $x \not\in \Omega$. 
\item[(H2)] $\Omega\subset \R^d_+$ is open and bounded with Lipschitz boundary.
\end{itemize}
In particular, since $\Omega$ is open (H1) implies that $f=f_*$ on $\R^d_+$.
The assumptions on $\Omega$ imply that the following cone condition is satisfied.
\begin{itemize}
\item[(H2$^*$)] For every $x \in \partial \Omega$, there exists a cone $\K_x$ with nonempty interior and a neighborhood $V_x$ of $x$ such that  
\[y\in V_x \setminus \Omega \implies (y + \K_x) \cap \bar{\Omega}\cap V_x \subset \{y\}.\]
\end{itemize}
To see this:  For any $x \in \partial \Omega$ there exists, by Lipchitzness of $\partial \Omega$, a real number $r>0$ and a Lipschitz continuous function $\Psi:\R^{d-1} \to \R$ such that, upon relabelling and reorienting the coordinate axes if necessary, we have
\[\Omega \cap B_r(x) = \{y \in B_r(x) \, : \, y_d < \Psi(y_1,\dots,y_{d-1})\}.\]
One can check that the cone
\[\K_x = \left\{y \in \R^d \, : \, y_d \geq 2{\rm Lip}(\Psi)\sqrt{y_1^2 + \cdots y_{d-1}^2}\right\}\]
satisfies (H2$^*$).
As it is more useful in the comparison principle proof, we will assume that (H2$^*$) holds instead of Lipschitzness of the boundary.
We note that the cone condition (H2$^*$) is similar to the one used by Deckelnick and Elliot~\cite[p.~331]{deckelnick2004}.  
\begin{theorem}\label{thm:comparison2}
Suppose that $\Omega$ satisfies (H2$^*$) and $f$ satisfies (H1).
Let $u$ and $v$ be viscosity sub- and supersolutions, respectively, of 
\begin{equation}\label{eq:pde-back}
u_{x_1} \cdots u_{x_d} = f \ \ \text{ on } \ \  \R^d_+,
\end{equation}
and assume that $u$ is truncatable and $v$ is Pareto-monotone.  Then $u\leq v$ on $\partial \R^d_+$ implies that $u\leq v$ on $\R^d_+$.
\end{theorem}

As in the proof of Theorem \ref{thm:comparison}, the proof below is based on the standard technique of doubling the variables~\cite{crandall1992}.   The proof is similar to~\cite[Theorem~2.3]{deckelnick2004} in the way that (H2$^*$) is used, however, we cannot assume Lipschitzness of $v$.  The truncatability condition on $u$ in a sense replaces the Lipschitz condition on $v$ in~\cite[Theorem~2.3]{deckelnick2004}.
\begin{proof}
For $\theta >0$, set $v_\theta(x) = v(x) + \theta^\frac{1}{d}\langle x, \vb{1}_d\rangle + \theta$.  Then $u < v_\theta$ on $\partial \R^d_+$.  As in the proof of Theorem \ref{thm:comparison}, $v_\theta$ is a viscosity supersolution of
\begin{equation}\label{eq:strict-super}
v_{\theta,x_1}\cdots v_{\theta,x_d} = f + \theta \ \ \text{ on } \ \  \R^d_+.
\end{equation}
Now suppose that $\sup_{\R^d_+} (u-v_\theta) >0$ and let
\[R = \sup \left\{ r>0 \, : \, u\leq v_\theta \ \ \text{ on } \ \  \R^d_+ \cap B_r(0)\right\}.\]
Since $u < v_\theta$ on $\partial \R_+^d$ and $u -v_\theta$ is continuous, we see that $0 < R< \infty$.  By the definition of $R$, there exists $z_0 \in \partial B_R(0) \cap \R^d_+$ such that $u(z_0)=v_\theta(z_0)$ and every neighborhood of $z_0$ contains a point $y$ such that $u(y) > v_\theta(y)$.
For $r>0$ set
\[H = \{ x \in \R^d \, : \, z_0 - r\vb{1}_d < x < z_0 + r^2 \vb{1}_d\}.\]
and note that $\sup_H (u-v_\theta) > 0$ for any $r>0$.

Notice that we may assume (H2$^*$) holds at any $x \in \R^d_+$.  Indeed, if $x \not\in \partial \Omega$ then we may set $V_x = B_\sigma(x)$ and choose $\sigma>0$ small enough so that $\partial \Omega \cap V_x = \varnothing$.  Then any cone $\K_x$  will suffice as either $V_x \setminus \Omega = \varnothing$ or $V_x \cap \bar{\Omega} = \varnothing$.  Let $\eta \in S^{d-1}$ be in the interior of $\K_{z_0}$.  For $x\in \R^d_+$ set
\[\tilde{u}(x_1,\dots,x_d)= u(\min(x_1,z_{0,1}+r^2),\dots,\min(x_d,z_{0,d} + r^2)).\]
For $\alpha > 0$ and $(x,y) \in \R^d_+\times \R^d_+$, set
\begin{equation}\label{eq:defPHI}
\Phi_\alpha(x,y) = \tilde{u}(x) - v_\theta(y) - \alpha\left|x - y - \frac{1}{\sqrt{\alpha}} \eta\right|^2,
\end{equation}
and $M_\alpha = \sup_{\R^d_+ \times \R^d_+} \Phi_\alpha$.  By continuity of $\tilde{u}$ and $v_\theta$ we have
\begin{equation}\label{eq:liminf}
\eps:= \liminf_{\alpha \to \infty} \ M_\alpha \geq \sup_{\R_+^d} (\tilde{u} - v_\theta) > 0.
\end{equation}
Set
\[D = \{x \in \R^d_+ \, : \, x_i \leq z_{0,i} - r \ {\rm for \ some} \ i\}.\]
Let $x \in D$ such that $x \leqq z_0 + r^2 \vb{1}_d$.  Then there exists $i$ such that $x_i \leq z_{0,i} -r$ and we have
\[|x|^2\leq (z_{0,i} - r)^2 + \sum_{j\neq i} (z_{0,j} + r^2)^2 = |z|^2 - 2z_{0,i} r + O(r^2).\]
Since $|z|=R$ and $z_{0,i} > 0$, we can choose $r>0$ small enough so that $x \in B_R(0)$.  Fixing such an $r>0$ we have
\begin{equation}\label{eq:uleqv}
\tilde{u} \leq v_\theta {\rm \ \ on \ \ } D \cap \{x \in \R^d_+ \, : \, x \leqq z_0 + r^2 \vb{1}_d\}.
\end{equation}

Since $\tilde{u}$ and $v_\theta$ are uniformly continuous on compact sets, it follows from \eqref{eq:uleqv} that there exists $\delta>0$ such that 
\[D \ni x \leqq z_0 + r^2\vb{1}_d \ \text{ and } \ |x-y| \leq \delta \implies \tilde{u}(x) - v_\theta(y) \leq \frac{\eps}{2},\]
for every $x,y \in \R^d_+$.
Now let $x\in D$ and $y \in \R^d_+$ with $|x-y| \leq \delta$, and set
\begin{align*}
\hat{x} &= (\min(x_1,z_{0,1}+r^2),\dots,\min(x_d,z_{0,d}+r^2)) \\
\hat{y} &= (\min(y_1,z_{0,1}+r^2),\dots,\min(y_d,z_{0,d}+r^2)).
\end{align*}
Then $\hat{x} \in D$, $\hat{x} \leqq z_0 + r^2\vb{1}_d$, $|\hat{x}-\hat{y}| \leq |x-y| \leq \delta$, $\tilde{u}(x) = \tilde{u}(\hat{x})$ and $v_\theta(y) \geq v_\theta(\hat{y})$.  We conclude that 
\[\tilde{u}(x) - v_\theta(y) \leq \tilde{u}(\hat{x}) - v_\theta(\hat{y}) \leq \frac{\eps}{2}.\]
Hence we have shown that for any $x,y\in\R^d_+$
\begin{equation}\label{eq:boundary-comp}
x \in D \ \text{ and } \ |x-y| \leq \delta \implies \tilde{u}(x) - v_\theta(y) \leq \frac{\eps}{2}.
\end{equation}

Fix $\alpha > 0$ large enough so that $M_\alpha > \eps/2$ and let $(x,y) \in \R^d_+ \times \R^d_+$ such that $\Phi_\alpha(x,y) > \eps/2$.  Then we have
\[\tilde{u}(x) - v(y) - \alpha\left| x-y-\frac{1}{\sqrt{\alpha}} \eta \right|^2 > \frac{\eps}{2}.\]
In particular, we have $\tilde{u}(x) - v_\theta(y) > \frac{\eps}{2}$, and
\[|x-y| \leq \frac{1}{\sqrt{\alpha}}\left( \left|\|\tilde{u}\|_{L^\infty(\R^d_+)} - v_\theta(0)\right|^\frac{1}{2} + 1\right).\]
Taking $\alpha $ large enough so that $|x-y| \leq \delta$, we see from \eqref{eq:boundary-comp} that $x \geqq z_0 - r\vb{1}_d$.
Now set
\[\hat{x} = (\min(x_1,z_{0,1}+r^2),\dots,\min(x_d,z_{0,d}+r^2)),\]
and $\hat{y} = y + \hat{x} - x$.  Then we have $\hat{x} \in \bar{H}$, $\hat{x}-\hat{y} = x-y$, $\tilde{u}(\hat{x})=\tilde{u}(x)$, and $v_\theta(\hat{y}) \leq v_\theta(y)$.  It follows that $\Phi_\alpha(\hat{x},\hat{y}) \geq \Phi_\alpha(x,y)$.  Hence for $\alpha>0$ large enough, $\Phi_\alpha$ attains a global maximum at $(x_\alpha,y_\alpha) \in \R^d_+\times \R^d_+$ satisfying $x_\alpha \in \bar{H}$ and $|x_\alpha - y_\alpha| \leq C/\sqrt{\alpha}$.  

Sending $\alpha\to \infty$ and extracting a subsequence, if necessary, we may assume that
\[x_\alpha \to x_0 \ \text{ and } \ y_\alpha \to x_0 \ \text{ as } \ \alpha \to \infty,\]
for some $x_0 \in \bar{H}$.
Since 
\[\tilde{u}(x_0) - v_\theta\left( x_0- \frac{1}{\sqrt{\alpha}} \eta\right) \leq M_\alpha \leq \tilde{u}(x_\alpha) - v_\theta(y_\alpha),\]
we see, by the continuity of $\tilde{u}$ and $v_\theta$, that
\[\lim_{\alpha\to\infty} M_\alpha = \tilde{u}(x_0) - v_\theta(x_0).\]
It follows that
\begin{equation}\label{eq:keybound-comp2}
\alpha\left|x_\alpha - y_\alpha - \frac{1}{\sqrt{\alpha}} \eta\right|^2 \to 0 \ \text{ as } \ \alpha \to \infty.
\end{equation} 

Let
\[p = 2\alpha\left(x_\alpha-y_\alpha - \frac{1}{\sqrt{\alpha}} \eta \right).\]
Since $(x_\alpha,y_\alpha) \in \R^d_+ \times \R^d_+$ is a local max of $\Phi_\alpha$, we have $p \in D^+ \tilde{u}(x_\alpha)\cap D^- v_\theta(y_\alpha)$.  Since $\tilde{u}$ is a truncatable, we have
\[p_1\cdots p_d \leq (\widehat{f}\;)^*(x_\alpha)  \leq f^*(x_\alpha),\]
where $\widehat{f}(x) = f^*(x)$ for $x \leqq z_0+r^2\vb{1}_d$ and $\widehat{f}(x)=0$ otherwise.  By \eqref{eq:strict-super} we have
\[p_1\cdots p_d \geq f_*(y_\alpha) + \theta.\]
Combining these we see that
\begin{equation}\label{eq:fcont-comp2}
f^{*}(x_\alpha) - f_{*}(y_\alpha) \geq \theta.
\end{equation}

If $x_0 \not\in \partial \Omega$, then for $\alpha$ large enough we have either $x_\alpha,y_\alpha \in \Omega$ or $x_\alpha,y_\alpha \in \R^d_+ \setminus \bar{\Omega}$.  Hence by (H1) we have
\begin{equation}\label{eq:key-eq}
f^*(x_\alpha) - f_*(y_\alpha) \leq m(|x_\alpha - y_\alpha|).
\end{equation}
Suppose now that $x_0 \in \partial \Omega$.  We have two cases; either  (1) $y_\alpha \in \Omega$ or (2) $y_\alpha \in \R^d_+\setminus \Omega$.  

Case 1. In this case we have $f_*(y_\alpha)=f(y_\alpha)$.  If $x_\alpha \in \bar{\Omega}$ then $f^*(x_\alpha)=f(x_\alpha)$ and hence \eqref{eq:key-eq} holds.
If $x_\alpha \in \R^d_+ \setminus \bar{\Omega}$ then $f^*(x_\alpha)=0$ and \eqref{eq:key-eq} holds trivially.

Case 2.  Set $w_\alpha = x_\alpha - y_\alpha - \frac{1}{\sqrt{\alpha}} \eta$ and note that
\[x_\alpha = y_\alpha + \frac{1}{\sqrt{\alpha}}(\eta + \sqrt{\alpha} w_\alpha).\]
By \eqref{eq:keybound-comp2} we have $\sqrt{\alpha}w_\alpha \to 0$ as $\alpha \to \infty$.  Since $\eta$ is in the interior of $\K_{z_0}$, we see that $\eta + \sqrt{\alpha}w_\alpha \in \K_{z_0}$ for $\alpha$ large enough, and hence $x_\alpha \in y_\alpha + \K_{z_0}$.  
We can take $r>0$ smaller, if necessary, so that $\bar{H} \in V_{z_0}$.  Since $x_\alpha \in \bar{H}$, we can choose $\alpha$ large enough so that $y_\alpha \in V_{z_0}$.  Since $x_\alpha \neq y_\alpha$ for $\alpha$ large enough, we have by (H2$^*$) that $x_\alpha \in \R^d_+\setminus \bar{\Omega}$ and hence $f^*(x_\alpha) = 0$ and \eqref{eq:key-eq} holds.
Sending $\alpha \to 0$ in \eqref{eq:key-eq} contradicts \eqref{eq:fcont-comp2}, hence $u \leq v_\theta$ on $\R^d_+$.  Sending $\theta \to 0$ we find that $u\leq v$ on $\R^d_+$.  
\end{proof}

In order to prove a general uniqueness result, Theorem \ref{thm:uniq}, we require a perturbation result for the value function $U$ with respect to sup and inf convolutions of the density $f$.  Since a similar result, for a different type of perturbation, is required in the proof of our main result, Theorem \ref{thm:linfty-conv}, we state a more general result in Lemma \ref{lem:l1-pert}.  We first recall some notation standard in the theory of viscosity solutions.  For a sequence of bounded functions $f_n: \R^d_+ \to \R$, the upper and lower limits are defined by
\[\limsups{n \to \infty}   f_n(x) := \lim_{j\to \infty} \sup\left\{f_n(y) \, : \, n\geq j, \ y \in \R^d, \ \text{ and } \ |x-y| \leq \frac{1}{j}\right\},\]
and
\[\liminfs{n \to \infty}   f_n(x) := \lim_{j\to \infty} \inf\left\{f_n(y) \, : \, n\geq j, \ y \in \R^d, \ \text{ and } \ |x-y| \leq \frac{1}{j}\right\}.\]
\begin{lemma}\label{lem:l1-pert}
Let $\{f_n\}_{n=1}^\infty \subset \B$ and suppose that $\Omega$ satisfies (H2$^*$), $f$ satisfies (H1), and 
\begin{equation}\label{eq:fn-cond}
f_* \leq \liminfs{n \to \infty} f_n \  \text{ and }  \ \limsups{n\to \infty} f_n \leq f^*.
\end{equation}
For each $n$, set 
\[v_n(x) =\sup_{\gamma \in \A \, : \, \gamma \leqq x} \int_0^1 f_n(\gamma(t))^\frac{1}{d} (\gam{1}'(t) \cdots \gam{d}'(t))^\frac{1}{d} \, dt.\]
Then $v_n \to U$ uniformly where $U$ is the value function given by \eqref{eq:vardef}. 
\end{lemma}
\begin{proof}
We claim that $\{f_n\}_{n=1}^\infty$ is a uniformly bounded sequence.  To see this, suppose to the contrary that there exists a sequence $x_n$ in $[0,1]^d$ such that $f_n(x_n) \to \infty$ as $n \to \infty$.  By passing to a subsequence, if necessary, we may assume that $x_n \to x_0 \in [0,1]^2$ as $n\to \infty$.  By the definition of the upper limit and \eqref{eq:fn-cond}, we have
\[f^*(x_0) \geq \limsups{n\to \infty} f_n(x_0) = \infty,\]
which contradicts the assumption that $f$ satisfies (H1) and establishes the claim.

Since $\{f_n\}_{n=1}^\infty$ is uniformly bounded, there exists (by Lemma \ref{lem:holder}) a constant $C$ such that $[v_n]_{\frac{1}{d}} \leq C$ for all $n$.  The sequence $v_n$ is therefore bounded and equicontinuous, and by the Arzela-Ascoli theorem there exists a subsequence $v_{n_k}$ and a H\"older-continuous function $v:\R^d \to \R$ such that $v_{n_k} \to v$ uniformly on compact sets in $\R^d$ as $k\to \infty$.  By Theorem \ref{thm:hjb} (i), (ii), we conclude that the convergence is actually uniform on $\R^d$.   By Theorem \ref{thm:hjb}, each $v_n$ is a Pareto-montone truncatable viscosity solution of 
\[v_{n,x_1} \cdots v_{n,x_d} =\frac{1}{d^d} f_n \ \ \text{ on } \ \  \R^d_+.\]
By standard results on viscosity solutions (see \cite[Remark 6.3]{crandall1992}) and \eqref{eq:fn-cond}, we have that $v$ is a Pareto-monotone viscosity solution of
\[v_{x_1} \cdots v_{x_d} =\frac{1}{d^d} f \ \ \text{ on } \ \  \R^d_+.\]
By the assumption that $\supp(f_n) \subset [0,1]^d$, we have that $v_n(x) = 0$ for all $x \not\in \R^d_+$, hence $v(x) = 0$ for all $x \not\in \R^d_+$.  

We claim that $v$ is truncatable. To see this, fix $z \in \R^d_+$ and define $\tilde{v}$, $\tilde{v}_n$, $\widehat{f}$ and $\widehat{f}_n$ as in Remark \ref{rem:truncation}.  Since $v_n$ is truncatable,  $\tilde{v}_n$ is a viscosity solution of
\[\tilde{v}_{n,x_1}\cdots \tilde{v}_{n,x_d} = \frac{1}{d^d} \widehat{f}_{n} \ \ \text{ on } \ \  \R^d_+.\]
By the definition of $\widehat{f}_n$, we have $\widehat{f}_n \leq f^*_n$, with $\widehat{f}_n(x) = f^*_n(x)$ for $x\leqq z$.  It follows that
\[\limsups{n\to\infty} \widehat{f}_n(x) \leq \limsups{n\to \infty} f_n^*(x) \leq f^*(x)=\widehat{f}(x)=(\widehat{f}\;)^*(x) \  \text{ for }  \ x \leqq z.\]
For $x \not\leqq z$, there exists a neighborhood $V$ of $x$ on which $\widehat{f}_n$ is identically zero for all $n$.  It follows that
\[\limsups{n\to\infty} \widehat{f}_n(x)  = 0 = (\widehat{f}\;)^*(x),\]
and therefore $\limsups{n\to \infty} \widehat{f}_n \leq (\widehat{f}\;)^*$.  
Since $\tilde{v}_{n_k} \to \tilde{v}$ uniformly, we can again apply standard results on viscosity solutions~\cite{crandall1992} to find that $\tilde{v}$ is a viscosity subsolution of
\[\tilde{v}_{x_1}\cdots \tilde{v}_{x_d} = \frac{1}{d^d}\widehat{f} \ \ \text{ on } \ \  \R^d_+,\]
which proves the claim.  

By Theorem \ref{thm:comparison2} we have $v=U$ on $\R^d_+$.  Since $U(x)=v(x) = 0$ for $x \not\in \R^d_+$ we have $v=U$ on $\R^d$.  The above argument can be used to show that every subsequence of $v_n$ contains a uniformly convergent subsequence converging to $U$.  It follows that $v_n \to U$ uniformly in $\R^d$ as $ n\to \infty$.
\end{proof}

We now establish uniqueness of viscosity solutions of \eqref{eq:hjb-main}.
\begin{theorem}\label{thm:uniq}
Suppose that $\Omega$ satisfies (H2$^*$) and $f$ satisfies (H1). Then there exists a unique Pareto-monotone viscosity solution $u$ of 
\begin{equation}\label{eq:pde-uniq}
\begin{cases}
u_{x_1} \cdots  u_{x_d} = f & \text{ on } \R^d_+ \\
u= 0 & \text{ on } \partial \R^d_+,
\end{cases}
\end{equation}
satisfying the additional boundary condition
\begin{equation}\label{eq:bdy}
u(x) = u(\min(x_1,1),\dots,\min(x_d,1)) \ \ \text{ for } \ \ x \in \R^d_+.
\end{equation}
\end{theorem}
\begin{proof}
By Theorem \ref{thm:hjb} there exists a H\"older-continuous Pareto-monotone viscosity solution $u$ of \eqref{eq:pde-uniq}.  To prove uniqueness, we will show that $u=d\cdot U$, where $U$ is the value function defined by \eqref{eq:vardef}.  Let $\eps>0$ and consider the inf and sup convolutions of $f$, defined for $x \in \R^d_+$ by
\[f_\eps(x) = \inf_{y \in \R^d_+} \left\{ f(y) + \frac{1}{\eps} |x-y| \right\}\ \ \text{ and } \ \ f^\eps(x) = \sup_{y \in \R^d_+} \left\{f(y) - \frac{1}{\eps}|x-y| \right\}.\]
Recall that $f_\eps$ and $f^\eps$ are Lipschitz continuous with constant $1/\eps$ and $f_\eps \leq f \leq f^\eps$.  
Without loss of generality, we may assume that $\bar{\Omega} \subset (0,1)^d$, and hence for $\eps > 0$ small enough, we have $\supp(f_\eps), \supp(f^\eps) \subset [0,1]^d$.  For $x \in \bar{\R^d_+}$, set
\[U^\eps(x) = \sup_{\gamma \in \A \, : \, \gamma \leqq x} \int_0^1 f^\eps(\gamma(t))^\frac{1}{d} (\gamma_1'(t) \cdots \gamma_d'(t))^\frac{1}{d} \, dt,\]
and
\[U_\eps(x) = \sup_{\gamma \in \A \, : \, \gamma \leqq x} \int_0^1 f_\eps(\gamma(t))^\frac{1}{d} (\gamma_1'(t) \cdots \gamma_d'(t))^\frac{1}{d} \, dt.\]
  By Theorem \ref{thm:hjb}, $d\cdot U^\eps$ is a viscosity solution of
\begin{equation}\label{eq:feps-pde}
v_{x_1} \cdots v_{x_d} = f^\eps \ \ \text{ on } \ \ \R^d_+,
\end{equation}
and satisfies the boundary condition \eqref{eq:bdy}.
Since $f \leq f^\eps$ and $u$ is a viscosity solution of \eqref{eq:pde-uniq}, we see that $u$ is a viscosity subsolution of \eqref{eq:feps-pde}.  Since $u=U = 0$ on $\partial  \R^d_+$ and $u$ satisfies \eqref{eq:bdy}, we can apply Theorem \ref{thm:comparison} to find that $u \leq d\cdot U^\eps$.  By a similar argument, we have that $u \geq d\cdot U_\eps $. Since $f_\eps,f^\eps \in \B$ and \eqref{eq:fn-cond} is satisfied for the sequences $\{f^\eps\}_{\eps>0}$ and $\{f_\eps\}_{\eps>0}$, we have by Lemma \ref{lem:l1-pert} that $U_\eps,U^\eps \to U$ uniformly in $\R^d$ as $\eps \to 0$, and hence $u=d\cdot U$.
\end{proof}

\section{Large sample asymptotics of $u_n$}
\label{sec:proof}

The proof of Theorem \ref{thm:linfty-conv} is split into several steps.  In Section \ref{sec:piecewise}, we prove a basic convergence result for piecewise constant density functions, which is a generalization of the results of Deuschel and Zeitouni~\cite{deuschel1995}.  In Section \ref{sec:cont}, we extend the convergence result to densities that are continuous on $\Omega$ and vanish on $\R^d\setminus \Omega$ by considering a sequence of piecewise constant approximations to $f$, applying the results from Section \ref{sec:piecewise}, and passing to the limit.  This requires a perturbation result for the energy $J$, which we obtained from the comparison principle for the associated Hamilton-Jacobi PDE \eqref{eq:hjb-main} in Lemma \ref{lem:l1-pert}.

\subsection{Piecewise constant densities}
\label{sec:piecewise}

We aim to prove a basic convergence result for piecewise constant densities here.
The proof is split into a lower bound, Theorem \ref{thm:lower-bound}, and an upper bound, Theorem \ref{thm:upper-bound}. We should note that the techniques used here are similar to those used by Deuschel and Zeitouni~\cite{deuschel1995}, who showed the same convergence result for $C^1$ densities on the unit hypercube in dimension $d=2$.  

Let us introduce some notation. For a finite set $S\subset \R^d$, let $\ell(S)$ denote the length of a longest increasing chain in $S$. The set function $\ell$ has an important invariance.  If $\Psi: \R^d \to \R^d$ is a mapping that preserves the partial order $\leqq$, i.e., $x \leqq y \iff \Psi(x) \leq \Psi(y)$, then
\begin{equation}\label{eq:ell-invariant}
\ell(S) = \ell(\Psi(S)) {\rm \ for \ any \ } S \subset \R^d.
\end{equation}
For $A\subset \R^d$ we denote by $\chi_{A}:\R^d\to \R$ the characteristic function of the set $A$, which takes the value $1$ on $A$ and $0$ on $\R^d\setminus A$.  When $A$ is Lebesgue measurable, we denote by $|A|$ the Lebesgue measure of $A$.   We set $\vb{0}_d=(0,\dots,0)\in \R^d$ and $\vb{1}_d = (1,\dots,1)\in \R^d$.  Given an integer $L$, we partition $[0,1)^d$ into $L^d$ hypercubes of side length $1/L$.  More precisely, for a multiindex $\alpha \in \N^L$ with $\|\alpha\|_\infty \leq L$, where $\|\alpha\|_\infty = \max(\alpha_1,\dots,\alpha_L)$, we set
\begin{equation}
Q_{L,\alpha} = \{x \in [0,1)^d \, : \,  \alpha - \vb{1}_d \leqq Lx < \alpha \}.
\end{equation}
We say that $f:[0,1)^d \to [0,\infty)$ is $L$-\emph{piecewise constant} if $f$ is constant on $Q_{L,\alpha}$ for all $\alpha$. 
If $f$ is $L$-piecewise constant then $f$ is $kL$-piecewise constant for all $k\in \N$.  
For convenience, we also set
\[\bar{J} = \sup_{\gamma \in \A} J(\gamma) = \sup_{\gamma \in \A} \int_0^1 f(\gamma(t))^\frac{1}{d} (\gam{1}'(t) \cdots \gam{d}'(t))^\frac{1}{d} \, dt.\]
We now establish an asymptotic lower bound on $\ell(\{X_1,\dots,X_n\})$.  
\begin{theorem}\label{thm:lower-bound}
Let $f:[0,1)^d \to [0,\infty)$ be $L$-piecewise constant, and let $X_1,\dots,X_n$ be \iid~with density $f$.  Then
\[\liminf_{n\to \infty}\ n^{-\frac{1}{d}} \ell\left(\{X_1,\dots,X_n\}\right) \geq c_d \bar{J} \almostsurely\]
\end{theorem}
\begin{figure}
\centering
\includegraphics[width=0.6\textwidth]{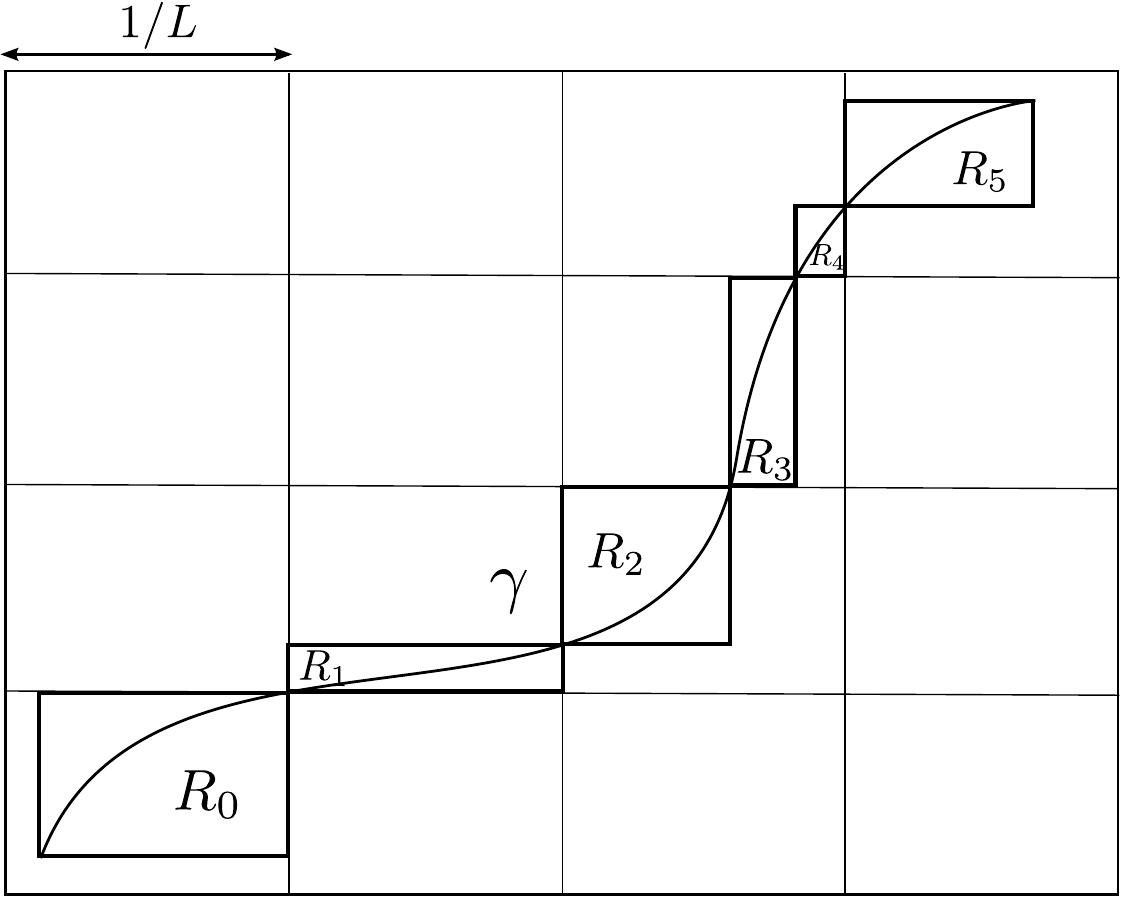}
\caption{An illustration of some quantities from the proof of Theorem \ref{thm:lower-bound}. }
\label{fig:lower-bound}
\end{figure}
\begin{proof}
Let $\eps>0$ and select $\gamma \in \A$ with $J(\gamma) \geq \bar{J} - \frac{\eps}{c_d}$.  Without loss of generality, we may assume that ${\gamma \:}'(t) > 0$ for all $t \in [0,1]$. Let $s_1,\dots,s_k$ denote the $k\leq dL$ times at which $\gamma$ intersects the set
\[\bigcup_{\alpha}\partial Q_{L,\alpha} \cap (0,1)^d.\]
Set $s_0=0$ and $s_{k+1}=1$.  For $j=0,\dots,k$ set $I_j=[s_j,s_{j+1})$ and
\begin{equation}\label{eq:Rj}
R_j = \{x \in [0,1)^d \, : \, \gamma(s_j) \leqq x < \gamma(s_{j+1})\}.
\end{equation}
For every $j$ we have $R_j \subset Q_{L,\alpha}$ for some $\alpha$.  Recalling the definition of $J$ \eqref{eq:Jdef} we have
\begin{equation}\label{eq:first}
J(\gamma) = \sum_{j=0}^{k} \int_{I_j} f(\gamma(t))^\frac{1}{d} (\gam{1}'(t) \cdots \gam{d}'(t))^\frac{1}{d} \, dt  
=\sum_{j=0}^{k} f(\gamma(s_j))^\frac{1}{d} \int_{I_j}(\gam{1}'(t) \cdots \gam{d}'(t))^\frac{1}{d} \, dt,
\end{equation}
where the second equality follows from the fact that $f$ is constant on $R_j \subset Q_{L,\alpha}$.  Applying the generalized H\"older inequality to \eqref{eq:first} we have
\begin{equation*}
J(\gamma) \leq\sum_{j=0}^{k} f(\gamma(s_j))^\frac{1}{d} \prod_{i=1}^d \left( \int_{I_j} \gami'(t) \, dt\right)^\frac{1}{d} 
= \sum_{j=0}^{k}f(\gamma(s_j))^\frac{1}{d} |R_j|^\frac{1}{d}.
\end{equation*}
Setting $p_j = \int_{R_j} f(x) \, dx= f(\gamma(s_j))|R_j|$ we have 
\begin{equation}\label{eq:pj}
\bar{J} \leq J(\gamma) + \frac{\eps}{c_d} \leq  \sum_{j =0}^{k} p_j^\frac{1}{d} + \frac{\eps}{c_d}.
\end{equation}

Fix $j \in\{0,\dots,k\}$.  Let $n_j$ denote the number of points from $X_1,\dots,X_n$ falling in $R_j$ and set
\begin{equation}\label{eq:ellj-def}
\ell_j(n) = \ell\left(\{X_1,\dots,X_n\} \cap R_j\right).
\end{equation}
Then $n_j$ is Binomially distributed with parameters $n$ and $p_j$. 
If $f$ is identically zero on $R_j$ then $\ell_j(n) = 0$ with probability one for all $n$, and $p_j=0$, hence
\begin{equation}\label{eq:trivial}
n^{-\frac{1}{d}} \ell_j(n) = c_dp_j^\frac{1}{d} \almostsurely 
\end{equation}
If $f$ is not identically zero on $R_j$, then since ${\gamma \:}'(t) > 0$ for all $t$, we have $|R_j| > 0$ and hence $p_j > 0$.  The conditional law $\rho_j := p_j^{-1} \cdot f \cdot \chi_{R_j}$ is then uniform on $R_j$.  Let $i_1,\dots,i_{n_j}$ be the indices of the $n_j$ random variables out of $X_1,\dots,X_n$ that belong to $R_j$.  Let  $\Psi:R_j \to [0,1)^d$ be the injective affine transformation mapping $R_j$ onto $[0,1)^d$.  Then $\Psi(X_{i_1}),\dots,\Psi(X_{i_{n_j}})$ are independent and uniformly distributed on $[0,1)^d$.   By~\cite[Remark 1]{bollobas1988}, we have 
\[n_j^{-\frac{1}{d}} \ell\left(\{\Psi(X_{i_1}),\dots,\Psi(X_{i_{n_j}})\}\right) \to c_d \almostsurely\]
Since $\Psi$ preserves the partial order $\leqq$, we have by \eqref{eq:ell-invariant} that
\[ \ell_j(n) = \ell\left(\{X_{i_1},\dots,X_{i_{n_j}}\}\right)=\ell\left(\{\Psi(X_{i_1}),\dots,\Psi(X_{i_{n_j}})\}\right).\]
Since $n^{-1} n_j \to p_j$ almost surely we have
\begin{equation}\label{eq:nontrivial}
n^{-\frac{1}{d}} \ell_j(n) = n_j^{-\frac{1}{d}} (n^{-1} n_j)^\frac{1}{d} \ell_j(n) \to c_d p_j^\frac{1}{d} \almostsurely
\end{equation}
Combining this \eqref{eq:pj}, \eqref{eq:trivial} and \eqref{eq:nontrivial}, we see that
\begin{equation}\label{eq:sum-conv}
n^{-\frac{1}{d}}\sum_{j =0}^{k} \ell_j(n)  \to c_d \sum_{j=0}^{k} p_j^\frac{1}{d} \geq c_d\bar{J} - \eps \almostsurely 
\end{equation}
Since $\gamma$ is a monotone curve (i.e., ${\gamma \:}'(t) \geq 0$), we can connect longest chains from each rectangle $R_j$ together to form a chain in $[0,1)^d$.  It follows that
\begin{equation}\label{eq:keyprop}
\ell\left(\{X_1,\dots,X_n\}\right) \geq \sum_{j =0}^{k} \ell_j(n).
\end{equation}
Combining this with \eqref{eq:sum-conv} we have
\begin{equation*}
\liminf_{n \to \infty} \ n^{-\frac{1}{d}}\ell\left(\{X_1,\dots,X_n\}\right)  \geq c_d\bar{J} - \eps \almostsurely,
\end{equation*}
which completes the proof. 
\end{proof}

For the proof of the upper bound, we need to introduce some new notation.  Let $k_1$ be an integer and set $\Delta x = 1/k_1$. 
Let $k_2$ be another integer and set $\Delta y = \Delta x / k_2$.  For given $k_1,k_2$, we say that a sequence of multiindices $\vb{b}=(b_j)_{j=1}^{k_1} \subset \N^{d-1}$ is \emph{admissible} if $b_1 \leqq \cdots \leqq b_{k_1}$ and $\|b_j\|_\infty \leq k_1 k_2$ for all $j$.  We denote the set of admissible multiindices by $\Phi(k_1,k_2)$.  For $\b \in \Phi(k_1,k_2)$, define $z_{\b,0},z_{\b,1},\dots,z_{\b,k_1}$ in $[0,1]^d$ by $z_{\b,0} = \vb{0}_d$ and $z_{\b,j} = (b_j\Delta y,j\Delta x)$ for $j\geq 1$.  Since $\b$ is admissible, $z_{\b,0},\dots,z_{\b,k_1}$ defines a chain in $[0,1]^d$.  Define $\gamma_\b:[0,1] \to [0,1]^d$ to be the polygonal curve connecting the points $z_{\b,0},\dots,z_{\b,k_1}$, i.e.,
\[\gamma_\b(t) = z_{\b,j-1} +  \frac{1}{\Delta x} (z_{\b,j}-z_{\b,j-1})(t- (j-1)\Delta x)\]
for $t \in [(j-1)\Delta x,j\Delta x]$.  For $\b \in \Phi(k_1,k_2)$ and $1 \leq j \leq k_1$, set
\[R_{\b,j} = \left\{x \in [0,1)^d \, : \, z_{\b,j-1} - (\vb{1}_{d-1},0)\Delta y \leqq x < z_{\b,j}\right\}.\]
For each rectangle $R_{\b,j}$, we set $p_{\b,j} = \int_{R_{\b,j}} f(x) \, dx$.
We say that a chain $x_1 \leqq x_2 \leqq \cdots \leqq x_m$ in $[0,1)^d$ is $\b$-\emph{increasing} if 
\[\{x_1,\dots,x_m\} \subset \bigcup_{j=1}^{k_1} R_{\b,j}.\]
It is not hard to see that for any $k_1,k_2$, every chain in $[0,1)^d$ is $\b$-increasing for some $\b \in \Phi(k_1,k_2)$.  See Figure \ref{fig:grid-example} for an illustration of the above definitions.   

\begin{figure}\label{fig:grid-example}
\centering
\includegraphics[width=0.75\textwidth]{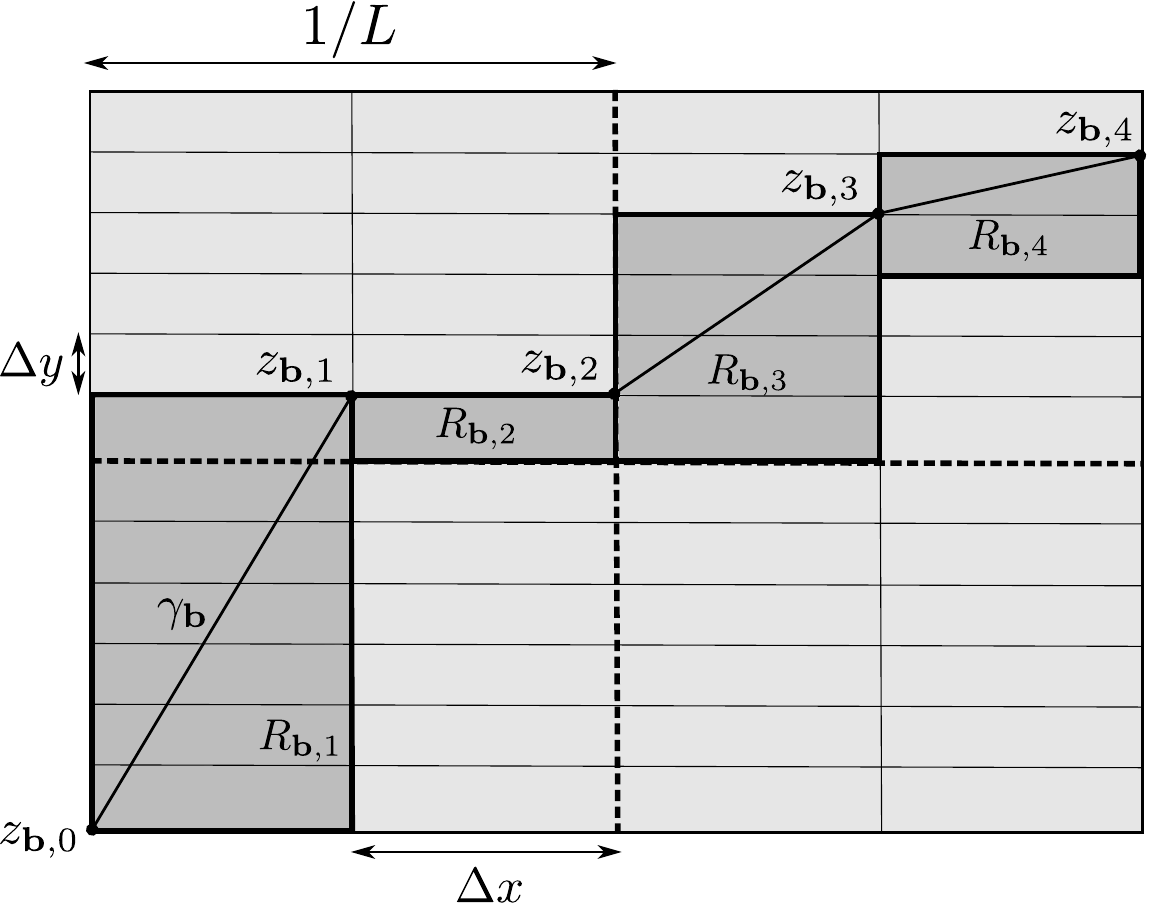}
\caption{An illustration of the quantities $R_{\b,j}$, $z_{\b,j}$, $Q_{L,\alpha}$ and $\gamma_\b$ in two dimensions with $\b=(b_1,b_2,b_3,b_4)=(7,7,10,11)$.  In this case, the unit square is partitioned into four squares, $Q_{L,(1,1)}$, $Q_{L,(1,2)}$, $Q_{L,(2,1)}$ and $Q_{L,(2,2)}$, which are separated by dotted lines in the figure.}
\end{figure}

We first need a preliminary lemma which bounds the length of a longest chain within the narrow strip
\begin{equation}\label{eq:T_j}
T_j := [0,1]^{d-1} \times [(j-1)\Delta x, j\Delta x),
\end{equation}
for any $j \in \{1,\dots,k_1\}$.
We note that the following lemma is a generalization of~\cite[Lemma 7]{deuschel1995}.  The proof is based on the same idea of using a mixing process to embed $X_1,\dots,X_n$ into another set of \iid~random variables that are uniform when restricted to the strip $T_j$.
\begin{lemma}\label{lem:strip}
Let $f:[0,1)^d \to [0,\infty)$ be $L$-piecewise constant, and let $X_1,\dots,X_n$ be \iid~with density $f$.
Fix an integer $j \in \{1,\dots,k_1\}$ and let $0 < \Delta x \leq \|f\|^{-1}_{L^\infty((0,1)^d)}$.
Then 
\begin{equation}\label{eq:strip-estimate}
\limsup_{n\to\infty} \  n^{-\frac{1}{d}}\ell\left(\{X_1,\dots,X_n\} \cap T_j\right) \leq c_d\left(2\Delta x \|f\|_{L^\infty((0,1)^d)}\right)^\frac{1}{d} \almostsurely
\end{equation}
\end{lemma}
\begin{proof}
Set $M=\|f\|_{L^\infty((0,1)^d)}$ and let $g = f + (M - f)\cdot \chi_{T_j}$.  Let $Y_1,\dots,Y_n$ be \iid~according to the conditional density $\beta^{-1} (M-f)\cdot \chi_{T_j}$ where $\beta = \int_{T_j} M - f(x) \, dx$.  Let $m_1,\dots,m_n$ be Bernoulli zero-one random variables with parameter $(1+\beta)^{-1}$ and set 
\begin{equation}\label{eq:index}
i_k = m_1 + \cdots +  m_k.
\end{equation}
Define $Z_1,\dots,Z_n$ through the mixture process
\[Z_k = m_k X_{i_k} + (1-m_k)Y_k.\]
Then $Z_1,\dots,Z_n$ are \iid~with density $(1+\beta)^{-1}g$.  Let $W$ denote the cardinality of the set $\{Z_1,\dots,Z_n\}\cap T_j$.  Then $W$ is binomially distributed with parameters $n$ and $p:=(1+\beta)^{-1}\Delta x M$.   Since $g$ is constant on $T_j$, we can use a similar argument to that in Theorem \ref{thm:lower-bound} to show that
\begin{equation}\label{eq:zlimit}
n^{-\frac{1}{d}}\ell\left(\{Z_1,\dots,Z_n\} \cap T_j\right) \to  c_d p^\frac{1}{d} \almostsurely
\end{equation}
Let $m=i_n$ and note that
\[\ell\left(\{X_1,\dots,X_{m}\}\right) = \ell\left(\{Z_k \, : \, m_k = 1\} \cap T_j\right) \leq \ell\left(\{Z_1,\dots,Z_n\}\cap T_j\right),\]
and that $p \leq \Delta x M$.  Combining this with \eqref{eq:zlimit} we have
\begin{equation}\label{eq:strip-a}
\limsup_{n\to \infty} \ n^{-\frac{1}{d}}\ell\left(\{X_1,\dots,X_{m}\}\right) \leq c_d (\Delta x M)^\frac{1}{d} \almostsurely 
\end{equation}
Since $m$ is Binomially distributed with parameters $n$ and $(1+\beta)^{-1}$, we have $n m^{-1} \to 1 + \beta$ almost surely and hence
\begin{align*}
\limsup_{n\to \infty} \ m^{-\frac{1}{d}} \ell\left(\{X_1,\dots,X_m\}\right) &{}={} \limsup_{n\to \infty} \ (n m^{-1})^\frac{1}{d} n^{-\frac{1}{d}} \ell\left(\{X_1,\dots,X_m\}\right) \\
&{}\leq{} (1+\beta)^\frac{1}{d} c_d (\Delta x M)^\frac{1}{d} \almostsurely
\end{align*}
Since $\beta \leq \Delta x M \leq 1$ we have 
\[\limsup_{n\to \infty} \ m^{-\frac{1}{d}} \ell\left(\{X_1,\dots,X_m\}\right) \leq c_d(2\Delta x M)^\frac{1}{d} \almostsurely\]
The desired result \eqref{eq:strip-estimate} follows from noting that $n\mapsto m(n)$ is monotone nondecreasing along every sample path and $m\to \infty$ as $n\to \infty$ with probability one.
\end{proof}

The following short technical lemma is essential in the proof of Theorem \ref{thm:upper-bound}
\begin{lemma}\label{lem:approx}
Let $f:[0,1)^d \to [0,\infty)$ be $L$-piecewise constant.  For every $\eps >0$ and $k_1\geq L$ we have
\begin{equation}\label{eq:Jb_bound}
\sum_{j\in \H_{\b}} p_{\b,j}^\frac{1}{d} \leq \bar{J} + \eps.
\end{equation}
for all $\b \in \Phi(k_1,k_2)$, the admissible multiindices, and $k_2\geq C \|f\|_{L^\infty((0,1)^d)}k_1^{d-1}/\eps^d$, where
\begin{equation}\label{eq:Hb}
\H_{\b} = \{j \, : \,  R_{\b,j} \subset Q_{L,\alpha} \ {\rm for \ some \ } \alpha\}.
\end{equation}
\end{lemma}
\begin{proof}
Let $k_1,k_2,\eps >0$, and $\b \in \Phi(k_1,k_2)$. Set $I_j = [(j-1)\Delta x, j\Delta x)$ and fix $j\in \{1,\dots,k_1\}$ and $t \in I_j$.  Note that
\[|R_{\b,j}| =\begin{cases}
\Delta y^{d-1} \Delta x \prod_{i=1}^{d-1} (b_{j,i}-b_{j-1,i} + 1)& {\rm if \ } j\geq 2 \\
\Delta y^{d-1} \Delta x \prod_{i=1}^{d-1} (b_{j,i}-b_{j-1,i})& {\rm if \ } j=1, \end{cases}\]
and 
\[\Delta x^d  \gamw{\b,1}(t) \cdots \gamw{\b,d}(t) = \Delta y^{d-1} \Delta x \prod_{i=1}^{d-1} (b_{j,i} - b_{j-1,i}),\] 
where we set $\b_0 = 0$ for convenience.  A short computation shows that
\begin{equation}\label{eq:rect-approx}
|\Delta x (\gamw{\b,1}(t) \cdots \gamw{\b,d}(t))^\frac{1}{d}  - |R_{\b,j}|^\frac{1}{d}| \leq C\Delta x^\frac{1}{d} \Delta y^\frac{1}{d},
\end{equation}
where $C=(d-1)^\frac{1}{d}$.
Since $f$ is $L$-piecewise constant we have
\begin{equation}\label{eq:const}
f(\gamma_\b(t)) = f(z_{\b,j-1}) = \frac{p_{\b,j}}{|R_{\b,j}|},
\end{equation}
for all $j \in \H_\b$ and $t \in I_j$.
Noting that $\Delta x = |I_j|$ and recalling the definition of $J$ \eqref{eq:Jdef} we have
\begin{align}\label{eq:Jb}
\bar{J} \geq J(\gamma_\b) &\stackrel{\eqref{eq:rect-approx}}{\geq} \sum_{j \in \H_{\b}}\frac{1}{|I_j|}  \int_{I_j} f(\gamma_\b(t))^\frac{1}{d} (|R_{\b,j}|^\frac{1}{d} - C\Delta x^\frac{1}{d} \Delta y^\frac{1}{d}) \, dt \notag\\
&\hspace{1mm}{}={} \sum_{j \in \H_{\b}}\frac{|R_{\b,j}|^\frac{1}{d}}{|I_j|}  \int_{I_j} f(\gamma_\b(t))^\frac{1}{d} \, dt - C \sum_{j \in H_\b}\frac{1}{|I_j|} \int_{I_j} f(\gamma_\b(t))^\frac{1}{d}\Delta x^\frac{1}{d} \Delta y^\frac{1}{d} \, dt \notag\\
&\stackrel{\eqref{eq:const}}{\geq}\sum_{j \in \H_{\b}} p_{\b,j}^\frac{1}{d}  - \|f\|_{L^\infty((0,1)^d)}^\frac{1}{d}k_1^\frac{d-1}{d} k_2^{-\frac{1}{d}}.
\end{align}
Taking $k_2\geq(C/\eps)^d\|f\|_{L^\infty((0,1)^d)} k_1^{d-1}$ completes the proof.
\end{proof}

We now establish an asymptotic upper bound on $\ell\left(\{X_1,\dots,X_n\}\right)$.
\begin{theorem}\label{thm:upper-bound}
Let $f:[0,1)^d \to [0,\infty)$ be $L$-piecewise constant, and let $X_1,\dots,X_n$ be \iid~with density $f$.
Then
\[\limsup_{n\to \infty} \ n^{-\frac{1}{d}}\ell\left(\{X_1,\dots,X_n\}\right) \leq c_d \bar{J} \almostsurely\]
\end{theorem}
\begin{proof}
Let $k_2>0,k_1\geq L, \eps>0$, and $\b \in \Phi(k_1,k_2)$.  We suppose that $k_1\geq L$ is a multiple of $L$ so that $f$ is $k_1$-piecewise constant. Let $\ell_\b(n)$ denote the length of a longest $\b$-increasing chain. Let $n_j$ denote the number of $X_1,\dots,X_n$ that belong $R_{\b,j}$ and set
\begin{equation}\label{eq:ellbj-def}
\ell_{\b,j}(n) = \ell\left(\{X_1,\dots,X_n\} \cap R_{\b,j}\right).
\end{equation}
Due to the monotonicity of $z_{\b,0},\dots,z_{\b,k_1}$, at most $(d-1)L$ of $R_{\b,1},\dots,R_{\b,k_1}$ can have a non-empty intersection with more than one hypercube $Q_{L,\alpha}$.  It follows that $|\H^c_{\b}| \leq (d-1) L$, where $\H^c_\b = \{1,\dots,k_1\}\setminus \H_\b$.

Since each $\b$-increasing chain is the union of chains in $R_{\b,1},\dots,R_{\b,k_1}$, we have
\begin{equation}\label{eq:ellb}
\ell_\b(n) \leq \sum_{j=1}^{k_1} \ell_{\b,j}(n) =  \sum_{j \in \H^c_\b} \ell_{\b,j}(n) + \sum_{j \in \H_\b} \ell_{\b,j}(n).
\end{equation}
We will deal with each of the above sums separately.  For the first term, set $M=\|f\|_{L^\infty((0,1)^d)}$ and let $k_1$ be large enough so that $\Delta x \leq 1/M$.  Since $R_{\b,j} \subset T_j$ for each $j$, we have by Lemma \ref{lem:strip} that
\[\limsup_{n \to \infty} \ n^{-\frac{1}{d}} \sum_{j \in \H^c_\b} \ell_{\b,j}(n) \leq c_d|\H^c_\b|(2\Delta x M)^\frac{1}{d} \leq c_d (d-1) L (2 M)^\frac{1}{d} k_1^{-\frac{1}{d}}\almostsurely\]
Choose $k_1$ large enough so that
\begin{equation}\label{eq:first-term}
\limsup_{n \to \infty} \ n^{-\frac{1}{d}} \sum_{j \in \H^c_\b} \ell_{\b,j}(n) \leq \frac{\eps}{2}\almostsurely 
\end{equation}

We now bound the second sum in \eqref{eq:ellb}. By Lemma \ref{lem:approx}, choose $k_2=k(M,k_1,\eps)$ so that
\begin{equation}\label{eq:sumpbj}
\sum_{j \in \H_\b} p_{\b,j}^\frac{1}{d} \leq \bar{J} + \frac{\eps}{2c_d},
\end{equation}
for all $\b\in \Phi(k_1,k_2)$.
For any $j \in \H_\b$, the conditional density $\rho_j$ on $R_{\b,j}$ is uniform.  By a similar argument as in the proof of Theorem \ref{thm:lower-bound}, we have that
\[n^{-\frac{1}{d}} \ell_{\b,j}(n)  \to c_d p_{\b,j}^\frac{1}{d} \almostsurely\]
Combining this with \eqref{eq:ellb}, \eqref{eq:first-term}, and \eqref{eq:sumpbj} we have
\begin{equation}\label{eq:conv-b}
\limsup_{n\to \infty} \ n^{-\frac{1}{d}} \ell_\b(n) \leq c_d\bar{J} + \eps \almostsurely
\end{equation}
Since every chain in $X_1,\dots,X_n$ is $\b$-increasing for some $\b \in \Phi(k_1,k_2)$, we have
\[\ell\left(\{X_1,\dots,X_n\}\right) \leq \max_{\b \in \Phi(k_1,k_2)} \ell_\b(n),\]
for every $n$. It follows that
\[\limsup_{n\to \infty} \ \ell\left( \{X_1,\dots,X_n\}\right) \leq c_d \bar{J} + \eps \almostsurely, \]
which completes the proof.
\end{proof}

\subsection{Continuous densities on $\Omega$}
\label{sec:cont}

We now generalize the convergence results on piecewise constant densities, Theorems \ref{thm:lower-bound} and \ref{thm:upper-bound}, to continuous densities on $\Omega$.  Our main result, Theorem \ref{thm:linfty-conv}, is proved at the end of the section.  The idea of our approach is to divide $[0,1)^d$ into a large number of hypercubes, and to flatten $f$ on each sub-cube.  We can then apply the results from Section \ref{sec:piecewise} and take the limit as the size of the sub-cubes tends to zero.  In order to pass to the limit, we apply the perturbation result given in Lemma \ref{lem:l1-pert}.

Let $X_1,\dots,X_n$ be \iid~with density $f$.  We recall that $u_n(x)$ denotes the length of a longest chain among $X_1,\dots,X_n$ consisting of points less than or equal to $x$ under the partial order $\leqq$.  In other words
\[u_n(x) = \ell\left(\{X_i \, : \, X_i \leqq x\}\right).\]
We also recall the definition of the value function $U$, defined in \eqref{eq:vardef} by
\[U(x) = \sup_{\gamma \in \A \, : \, \gamma \leqq x} \int_0^1 f(\gamma(t))^\frac{1}{d} (\gam{1}'(t) \cdots \gam{d}'(t))^\frac{1}{d} \, dt.\]
We now establish pointwise asymptotic upper and lower bounds on $u_n$.
\begin{theorem}\label{thm:pointwise-upper}
Let $f:\R^d \to \R$ satisfy (H1) and let $\Omega\subset \R^d_+$ satisfy (H2). Then for every $z \in \R^d$ we have
\begin{equation}\label{eq:pointwise-limit}
\limsup_{n\to \infty} \ n^{-\frac{1}{d}} u_n(z) \leq c_d U(z) \almostsurely 
\end{equation}
\end{theorem}
\begin{proof}
Set $D=\{x \in \R^d \, : \, 0 \leqq x < z\}$ and $p=\int_D f(x) \, dx$.  Suppose that $p=0$. It follows from (H1) that $f$ is lower semicontinuous, and hence $f(x)=0$ for $x \leqq z$.  Thus $u_n(z)=0=U(z)$ almost surely.

Suppose that $p>0$ and let $\eps > 0$. Let $k\in \N$ and partition $D$ into $k^d$ hypercubes $Q_{k,\alpha}$ for multiindices $\alpha$ with $\|\alpha\|_\infty \leq k$. Define $f_k:\R^d \to [0,\infty)$ by
\begin{equation}\label{eq:piecewise-sup}
f_{k}(x) = \sum_{\alpha} \Big(\sup_{Q_{k,\alpha}} f\Big)\chi_{Q_{k,\alpha}}(x) + f(x)\chi_{\R^d\setminus D}(x),
\end{equation}
and set $p_k = \int_D f_k(x) \, dx$. For every integer $k$, $f_k$ is $k$-piecewise constant on $D$ and $f \leq f_k$.  
Define $v_{k}: \R^d \to \R$ by
\begin{equation}\label{eq:vk}
v_{k}(x) = \sup_{\gamma \in \A \, : \, \gamma\leqq x} \int_0^1 f_{k}(\gamma(t))^\frac{1}{d} (\gam{1}'(t)\cdots \gam{d}'(t))^\frac{1}{d} \, dt.
\end{equation}
Note that the sequence $f_{k}$ is uniformly bounded, Borel-measurable, and has compact support in $[0,1]^d$. Furthermore, it follows from \eqref{eq:piecewise-sup} that \eqref{eq:fn-cond} holds for the sequence $f_k$.
Hence by Lemma \ref{lem:l1-pert} we have that $v_{k} \to U$ uniformly as $k \to \infty$.
Now fix $k$ large enough so that 
\begin{equation}\label{eq:limit}
|v_{k}(z) - U(z) | \leq \frac{\eps}{c_d}.
\end{equation}

Set
\begin{equation}\label{eq:lambda-pt}
\lambda = \left(\int_{\R^d} f_k(x) \, dx\right)^{-1},
\end{equation}
and define $g=\lambda f_k$.  Then $\lambda f \leq g$ and we can write $g$ as a convex combination of two distributions as follows:
\[g = \lambda f + (g-\lambda f).\]
Let $Y_1,\dots,Y_n$ be \iid~with density $(1-\lambda)^{-1} (g-\lambda f)$, let $m_1,\dots,m_n$ be Bernoulli random variables with parameter $\lambda$, and set
\[i_j = m_1+ \cdots + m_j.\]
 Define
\[Z_j = m_j X_{i_j} + (1-m_j) Y_j.\]
Then a simple computation shows that $Z_1,\dots,Z_n$ are \iid~with density $g$.  Let $W$ denote the cardinality of $\{Z_1,\dots,Z_n\}\cap D$.  Since $g$ is $k$-piecewise constant on $D$, we can apply Theorems \ref{thm:lower-bound} and \ref{thm:upper-bound} to obtain
\begin{equation}\label{eq:init-limit}
\lim_{n\to \infty} W^{-\frac{1}{d}} \ell\left(\{Z_1,\dots,Z_n\}\cap D\}\right) = c_d p_k^{-\frac{1}{d}} v_k(z) \almostsurely
\end{equation}
Note that $W$ is Binomially distributed with parameters $n$ and $\lambda p_k$, hence $n^{-1} W \to \lambda p_k$ almost surely.  Applying this to \eqref{eq:init-limit} we have
\begin{align}\label{eq:second-limit}
\lim_{n \to \infty} n^{-\frac{1}{d}} \ell\left(\{Z_1,\dots,Z_n \}\cap D\right) &{}={} \lim_{n\to \infty} \ (n^{-1} W)^\frac{1}{d} W^{-\frac{1}{d}} \ell\left(\{Z_1,\dots,Z_n \}\cap D\right) \notag \\
&{}={} c_d \lambda^\frac{1}{d} v_k(z)\almostsurely 
\end{align}
Set $m=i_n$.  Note that $m$ is Binomially distributed with parameters $n$ and $\lambda$, and 
\begin{equation}\label{eq:um}
u_m(z) = \ell\left( \{X_1,\dots,X_m\} \cap D \right) \leq \ell\left(\{Z_1,\dots,Z_n\}\cap D\right).
\end{equation}
Combining \eqref{eq:um} with \eqref{eq:second-limit} and the fact that $n^{-1} m \to \lambda$ as $n\to \infty$ we have
\begin{align}\label{eq:final-limit}
\limsup_{n\to \infty}  \ m^{-\frac{1}{d}} u_m(z) &{}\stackrel{\eqref{eq:um}}{\leq}{} \lim_{n\to \infty} \ (m^{-1} n)^{\frac{1}{d}}n^{-\frac{1}{d}}\ell\left(\{Z_1,\dots,Z_n\}\cap D\right) \notag \\
&{}\stackrel{\eqref{eq:second-limit}}{=} c_d v_k(z) \almostsurely
\end{align}
Recalling \eqref{eq:limit} we have
\[\limsup_{n\to \infty} \ m^{-\frac{1}{d}} u_m(z) \leq c_d U(z) + \eps \almostsurely\]
As in the proof of Lemma \ref{lem:strip}, the proof is completed by noting that $n\mapsto m(n)$ is monotone nondecreasing along every sample path and $m\to \infty$ as $n\to \infty$ with probability one.
\end{proof}

\begin{theorem}\label{thm:pointwise-lower}
Let $f:\R^d \to \R$ satisfy (H1) and let $\Omega\subset \R^d_+$ satisfy (H2). Then for every $z \in \R^d$ we have
\begin{equation}\label{eq:pointwise-limit2}
\liminf_{n\to \infty} \ n^{-\frac{1}{d}} u_n(z) \geq c_d U(z) \almostsurely
\end{equation}
\end{theorem}
\begin{proof}
Let $\eps>0$.  As in the proof of Theorem \ref{thm:pointwise-upper}, we set $D:=\{x \in \R^d \, : \, 0 \leqq x < z\}$ and we may suppose that $p:=\int_D f(x) \, dx >0$.  As before, let $k\in \N$ and partition $D$ into $k^d$ hypercubes $Q_{k,\alpha}$ for multiindices $\alpha$ with $\|\alpha\|_\infty \leq k$. Define $f_k:\R^d \to [0,\infty)$ by
\begin{equation}\label{eq:piecewise-inf}
f_{k}(x) = \sum_{\alpha} \Big(\inf_{Q_{k,\alpha}} f\Big)\chi_{Q_{k,\alpha}}(x) + f(x)\chi_{\R^d\setminus D}(x),
\end{equation}
and set $p_k = \int_D f_k(x) \, dx$. 
Define
\begin{equation}\label{eq:q}
q(x) =\displaystyle \begin{cases}
\frac{f_k(x)}{f(x)},& \text{if } f(x) > 0 \\
0,& \text{otherwise.}
\end{cases}
\end{equation}
For any $\alpha$ such that $Q_{k,\alpha} \subset \Omega$, we have by (H1) and \eqref{eq:piecewise-inf} that
\[f_k(x) \geq f(x) - m\left(\frac{\sqrt{d}}{k}\right) \text{ for } x \in Q_{k,\alpha}.\]
It follows that $\|q\|_{L^\infty(\R^d)} \nnearrow 1$ as $k\to \infty$.  As in the proof of Theorem \ref{thm:pointwise-upper}, we have that $v_{k} \to U$ uniformly as $k \to \infty$, where $v_k$ is defined by \eqref{eq:vk}.  We can therefore fix $k$ large enough so that 
\begin{equation}\label{eq:limit-lower}
c_d\|q\|^{-\frac{1}{d}}_{L^\infty(\R^d)} v_k(z) \geq c_d U(z)  - \eps.
\end{equation}

For $i=1,\dots,n$, let $m_i$ be a Bernoulli zero-one random variable with parameter $\|q\|^{-1}_{L^\infty(\R^d)} q(X_i)$.  Let $m=m_1+ \cdots + m_n$ and let $i_1,\dots,i_m$ denote the indices for which $m_i=1$. We claim that $X_{i_1},\dots,X_{i_m}$ are \iid~with density $g:=\lambda f_k$ where $\lambda$ is defined in \eqref{eq:lambda-pt}.
To see this, first note since $f(x)=0$ implies $f_k(x) = 0$, we have $q(x) f(x) = f_k(x)$ for all $x\in \R^d$.  Thus
\begin{equation}\label{eq:Pm1}
P(m_i = 1) = \int_{\R^d} P(m_i = 1\, | \, X_i=x) f(x) \, dx = \int_{\R^d} \frac{q(x)}{\|q\|} f(x) \, dx = \frac{1}{\lambda\|q\|},
\end{equation}
where $\|q\|=\|q\|_{L^\infty(\R^d)}$.  Let $j\geq 1$ and let $A \subset \R^d$ be measurable.  We have
\begin{align*}
P(X_{i_j} \in A) &{}={} P(X_i \in A \, | \, m_i = 1) \\
&{}={} \frac{P(X_i \in A \ \text{ and } \ m_i = 1)}{P(m_i=1)} \\
&{}\hspace{-0.75mm}\stackrel{\eqref{eq:Pm1}}{=}{}\lambda \|q\| \int_{A} P(m_i = 1 \, | \, X_i=x) f(x) \, dx \\
&{}={} \lambda\|q\| \int_{A} \frac{q(x)}{\|q\|} f(x) \, dx \\
&{}={} \int_{A} \lambda f_k(x) \, dx.
\end{align*}
By the construction of $X_{i_1},\dots,X_{i_m}$, they are independent random variables, hence the claim is established.

Let $W$ denote the cardinality of $\{X_{i_1},\dots,X_{i_m}\}\cap D$.  By Theorems \ref{thm:lower-bound} and \ref{thm:upper-bound}, we have
\begin{equation}\label{eq:lim1}
\lim_{n\to \infty} W^{-\frac{1}{d}} \ell\left(\{X_{i_1},\dots,X_{i_m}\}\cap D\right) = c_d p_k^{-\frac{1}{d}} v_k(z) \almostsurely
\end{equation}
Define
\[w_i = \begin{cases}
1,&\text{if } m_i = 1 \ \text{ and } \ X_i \in D\\
0,& \text{otherwise.}
\end{cases}\]
Then $W=w_1 + \cdots w_n$. Each $w_i$ is a Bernoulli zero-one random variable with parameter
\[P(w_i=1) = P(m_i =1  \ \text{ and } \ X_i \in D) = \int_D P(m_i = 1 \, | \, X_i = x) f(x) \, dx = \frac{p_k}{\|q\|}.\]
It follows that $W$ is Binomially distributed with parameters $n$ and $\|q\|^{-1} p_k$, and hence $n^{-1} W \to \|q\|^{-1} p_k$ almost surely. Combining this with \eqref{eq:lim1} yields
\begin{align}\label{eq:lim2}
\lim_{n\to \infty} n^{-\frac{1}{d}} \ell\left(\{X_{i_1},\dots,X_{i_m}\}\cap D\right) &{}={}\lim_{n\to \infty} \ (n^{-1} W)^\frac{1}{d} W^{-\frac{1}{d}} \ell\left(\{X_{i_1},\dots,X_{i_m}\}\cap D\right) \notag \\
&{}={}c_d \|q\|^{-\frac{1}{d}}  v_k(z) \almostsurely
\end{align}
Noting that
\[u_n(z) = \ell\left(\{X_1,\dots,X_n\}\cap D \right) \geq \ell\left(\{X_{i_1},\dots,X_{i_m}\}\cap D\right),\]
we have
\begin{align}\label{eq:lim3}
\liminf_{n\to \infty} \  n^{-\frac{1}{d}} u_n(z) &{}\geq{}\lim_{n\to \infty} n^{-\frac{1}{d}} \ell\left(\{X_{i_1},\dots,X_{i_m}\}\cap D\right)\notag \\
&\hspace{-1mm}\stackrel{\eqref{eq:lim2}}{=} c_d \|q\|^{-\frac{1}{d}}  v_k(z) \almostsurely
\end{align}
Recalling \eqref{eq:limit-lower} we have
\[\liminf_{n\to \infty} \ n^{-\frac{1}{d}} u_n(z) \geq c_d U(z) - \eps \almostsurely,\]
which completes the proof.
\end{proof}

We now have the proof of Theorem \ref{thm:linfty-conv}.
\begin{proof}
Let $\eps>0$. Let $k\in \N$ and for a multiindex $\alpha \in \Z^d$, set $x_\alpha = \alpha/k$. 
Since $U$ is uniformly continuous (by Lemma \ref{lem:holder}) we can choose $k$ large enough so that 
\begin{equation}\label{eq:smoothness_u}
|U(x_{\alpha+\vb{1}_d}) - U(x_{\alpha})| \leq \frac{\eps}{c_d}
\end{equation}
for all $\alpha\in\Z^d$.    Let $I$ be the set of multiindices $\alpha$ for which $x_\alpha \in [0,1]^d$. Note that the cardinality of $I$ is $(k+1)^d$.  Since $I$ is finite with cardinality independent of $n$, Theorems \ref{thm:pointwise-upper} and \ref{thm:pointwise-lower} yield
\begin{equation}\label{eq:sup-limit}
\lim_{n\to \infty}  \sup_{\alpha \in I} |n^{-\frac{1}{d}} u_n(x_\alpha) - c_d U(x_\alpha)| = 0 \almostsurely
\end{equation}
Let $z \in (0,1]^d$.  Then there exists $\alpha\in I$ such that $x_\alpha < z \leqq x_{\alpha + \vb{1}_d}$.  By the Pareto-monotonicity of $u_n$ and \eqref{eq:smoothness_u} we have
\[n^{-\frac{1}{d}}u_n(z) - c_dU(z) \leq n^{-\frac{1}{d}} u_n(x_{\alpha + \vb{1}_d}) - c_d U(z) \stackrel{\eqref{eq:smoothness_u}}{\leq} n^{-\frac{1}{d}} u_n(x_{\alpha + \vb{1}_d}) - c_d U(x_{\alpha + \vb{1}_d}) + \eps.\]
By a similar argument, we have
\[n^{-\frac{1}{d}}u_n(z) - c_dU(z) \geq n^{-\frac{1}{d}}u_n(x_{\alpha}) - c_d U(x_{\alpha}) - \eps,\]
and hence
\begin{equation}\label{eq:sup-bound}
\| n^{-\frac{1}{d}}u_n - c_dU\|_{L^\infty((0,1)^d)}  \leq \sup_{\alpha \in I} |n^{-\frac{1}{d}} u_n(x_\alpha) - c_d U(x_\alpha)| + \eps.
\end{equation}
Combining \eqref{eq:sup-limit} and \eqref{eq:sup-bound} we have
\[\limsup_{n\to \infty} \|n^{-\frac{1}{d}} u_n - c_d U \|_{L^\infty((0,1)^d)} \leq \eps \almostsurely,\]
and hence $\lim_{n\to \infty} \|n^{-\frac{1}{d}} u_n - c_d U \|_{L^\infty((0,1)^d)} = 0$ almost surely.
The desired result now follows immediately from the boundary conditions on $U$ proved in Theorem \ref{thm:hjb} (i), (ii) and the fact that there are almost surely no samples in $\R^d\setminus (0,1)^d$.
\end{proof}

As a straightforward application of Theorem \ref{thm:linfty-conv}, we can show that non-dominated sorting is stable under bounded random perturbations in the samples $X_1,\dots,X_n$.  For $\delta > 0$, we set 
\[Z_i = X_i + Y_i\delta \ \text{ for }  \ i=1,\dots,n,\]
where $Y_1,\dots,Y_n$ are \iid~with a continuous compactly supported density function $g:\R^d\to \R$.  For $x \in \R^d$, set
\[u^\delta_n(x) = \ell\left(\{Z_i \, : \, Z_i \leqq x\}\right).\] 
\begin{theorem}[Stability of non-dominated sorting]\label{thm:stability}
Let $f:\R^d \to \R$ satisfy (H1) and let $\Omega \subset \R^d_+$ satisfy (H2).  
There exist constants $C_\delta$, depending only on $\delta$, $f$, and $g$, such that
\[\limsup_{n\to \infty} \ n^{-\frac{1}{d}} \|u_n^\delta - u_n\|_{L^\infty(\R^d)} \leq C_\delta \almostsurely \]
and $C_\delta \to 0$ as $\delta \to 0$. 
\end{theorem}
\begin{proof}
Set $g_\delta(x) = \frac{1}{\delta^d} g\left(\frac{x}{\delta}\right)$.  Then $Z_1,\dots,Z_n$ are \iid~with density $f_\delta := g_\delta * f$.  Set 
\[U^\delta(x) = \sup_{\gamma \in \A \, : \, \gamma \leqq x} \int_0^1 f_\delta(\gamma(t))^\frac{1}{d}(\gam{1}'(t)\cdots \gam{d}'(t))^\frac{1}{d} \, dt.\]
Without loss of generality, we may suppose that $\bar{\Omega} \subset (0,1)^d$.  Since $\supp(f) \subset \bar{\Omega}$ and $g$ has compact support, we can take $\delta>0$ small enough so that $\supp(f_\delta) \subset [0,1]^d$.
It is not hard to see that \eqref{eq:fn-cond} holds for the sequence $f_{\delta}$. Since each $f_\delta$ is continuous and bounded with compact support in $[0,1]^d$, it follows from Lemma \ref{lem:l1-pert} that $U^\delta \to U$ uniformly on $\R^d$.  Note that
\begin{align*}
n^{-\frac{1}{d}} \|u_n^\delta - u_n\|_{L^\infty(\R^d)} \leq \|n^{-\frac{1}{d}} u_n^\delta - c_d U^\delta \|_{L^\infty(\R^d)} &+ \|n^{-\frac{1}{d}} u_n - c_d U\|_{L^\infty(\R^d)} \\
&+ c_d \|U^\delta - U\|_{L^\infty(\R^d)},
\end{align*}
for every $n$.
Since $f_\delta$ is continuous on $(0,1)^d$ and $f_\delta(x) = 0$ for $x \not\in (0,1)^d$, (H1) is satisfied for $f_\delta$ by taking $\Omega'=(0,1)^d$.  We can therefore apply Theorem \ref{thm:linfty-conv} to obtain
\[\|n^{-\frac{1}{d}} u_n^\delta - c_d U^\delta \|_{L^\infty(\R^d)} + \|n^{-\frac{1}{d}} u_n - c_d U\|_{L^\infty(\R^d)}  \to 0 \almostsurely\]
The proof is completed by setting $C_\delta = c_d \|U^\delta - U\|_{L^\infty(\R^d)}$.
\end{proof}

\section{Numerical demonstrations}
\label{sec:num}

Theorem \ref{thm:linfty-conv} guarantees that the level sets of $c_dU$ will provide good approximations to the Pareto fronts for large $n$.  In this section, we present a numerical scheme for computing $U$ and show examples comparing the level sets of $c_d U$ to the Pareto fronts for various density functions.  To compute $U$, we need to compute the Pareto-monotone viscosity solution of 
\begin{equation}\label{eq:pde-num}
U_{x_1} \cdots U_{x_d} = \frac{1}{d^d} f \ \  \text{on} \ \ \R^d_+,
\end{equation}
that satisfies $U=0$ on $\partial \R^d_+$.  Let $U_\alpha$ and $f_\alpha$ denote the values of $U$ and $f$ on a grid with spacing $\Delta x$, where $\alpha$ is a multi-index.    For a given grid point $\alpha$, the domain of dependence for \eqref{eq:pde-num} is $\{ \beta \, : \, \beta \leqq \alpha\}$.  Hence an upwind scheme will use backward difference quotients.  Now consider substituting backward difference quotients into \eqref{eq:pde-num}.  We have
\begin{equation}\label{eq:discretize1}
\prod_{i=1}^d (U_\alpha - U_{\alpha - e_i}) = \frac{\Delta x^d}{d^d} f_\alpha.
\end{equation}
We intend to solve the above equation for $U_\alpha$ in terms of $U_{\alpha-e_i}$.  
Since we intend to compute the Pareto-monotone solution of \eqref{eq:pde-num}, we should look for a solution with $U_\alpha \geq U_{\alpha - e_i}$ for all $i$.  Consider the mapping
\[p \mapsto \prod_{i=1}^d (p-a_i) - a_0,\]
where $a_i\geq 0$ for all $i$.  Note that this mapping is strictly increasing for $p> \max(a_1,\dots,a_d)$ and $\prod_{i=1}^d (p-a_i) = 0$ for $p = \max(a_1,\dots,a_d)$.  Hence, for any non-negative $a_0,\dots,a_d$, there exists a unique $p\geq \max(a_1,\dots,a_d)$ satisfying
\[\prod_{i=1}^d(p-a_i) = a_0.\]
We denote this solution $p$ by $P(a_0,\dots,a_d)$ and define our numerical scheme by
\begin{equation}\label{eq:discretize2}
U_\alpha = P(\Delta x^d f_\alpha/d^d, U_{\alpha-e_1},\dots,U_{\alpha-e_d})
\end{equation}
with the boundary condition $U_\alpha = 0$ for $\alpha$ with $\min(\alpha_1,\dots,\alpha_d)=0$.  The numerical solution can be computed by any sweeping pattern respecting the partial order $\leqq$.  This requires visiting each grid point exactly once, and hence has linear complexity.  
The value of $P(a_0,\dots,a_d)$ can be computed numerically by either a binary search and/or Newton's method restricted to the interval $[\max(a_1,\dots,a_d), \max(a_1,\dots,a_d)+a_0^{1/d}]$.
In the case of $d=2$, we have a closed form expression for $P$
\[P(a_0,a_1,a_2) = \frac{1}{2}(a_1+a_2) + \frac{1}{2}\sqrt{(a_1-a_2)^2 + 4a_0}.\]
Note that we have chosen the positive square root to obtain the Pareto-monotone solution.  We prove in a subsequent paper~\cite{calder2013b}, that the numerical solutions, defined as above, converge to the unique Pareto-monotone viscosity solution of \eqref{eq:pde-num} as $\Delta x \to 0$.  

\begin{figure}[t!]
\centering
\subfigure{\includegraphics[clip = true, trim = 30 25 60 20, width=0.475\textwidth, height=0.255\textheight]{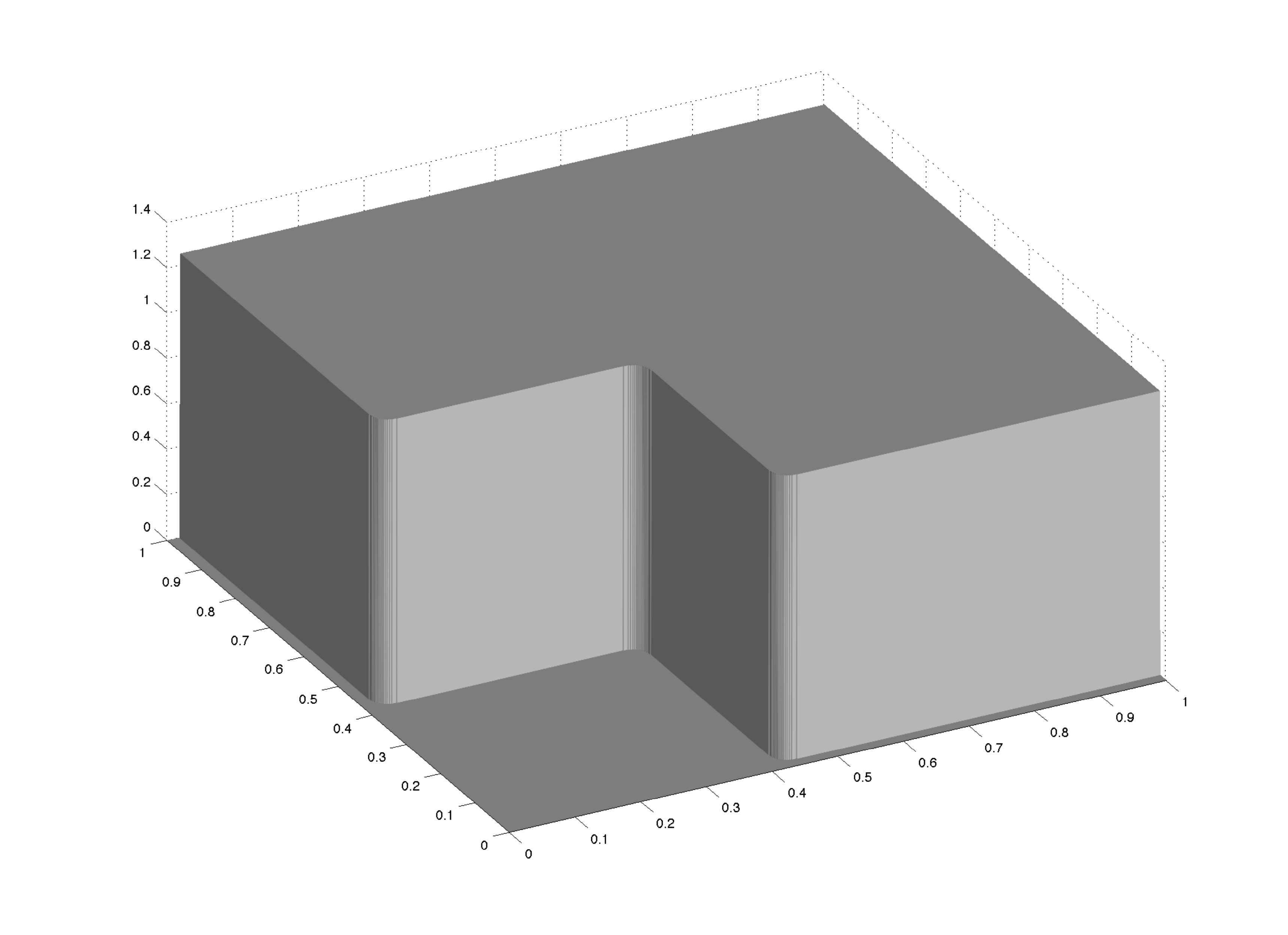}}
\subfigure{\includegraphics[clip = true, trim = 30 25 40 20, width=0.475\textwidth]{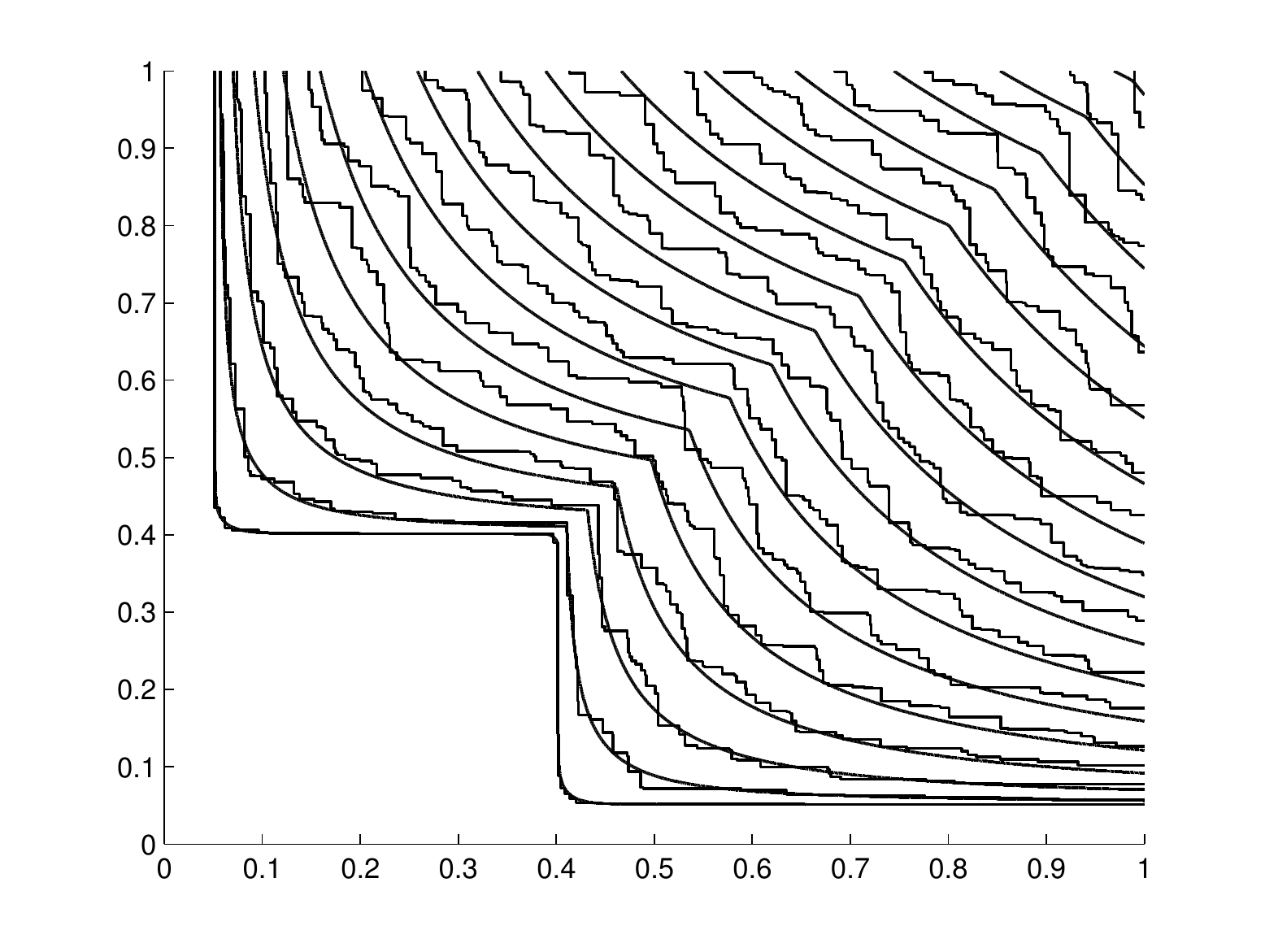}}\\
\subfigure{\includegraphics[clip = true, trim = 30 25 40 20, width=0.475\textwidth]{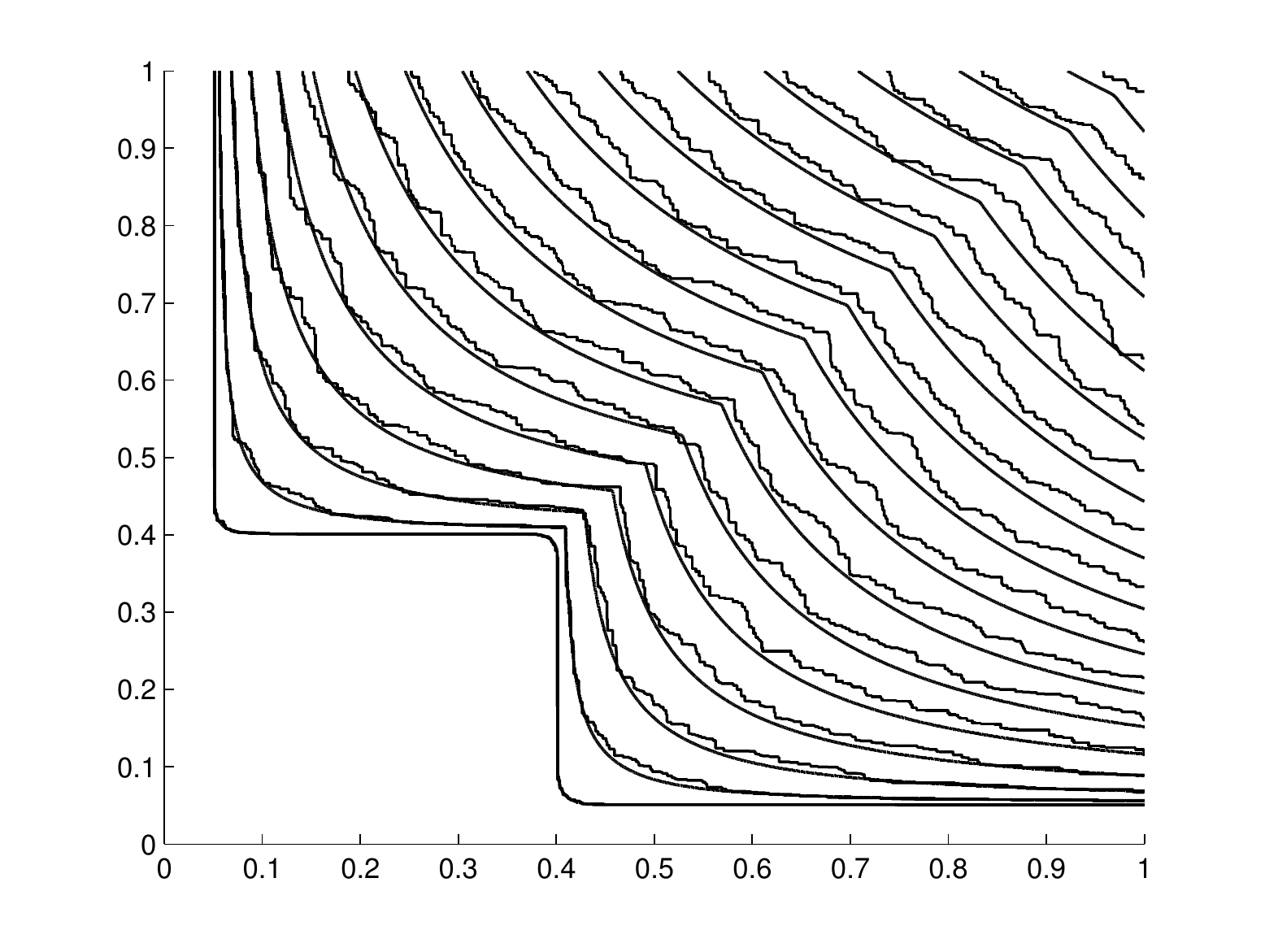}}
\subfigure{\includegraphics[clip = true, trim = 30 25 40 20, width=0.475\textwidth]{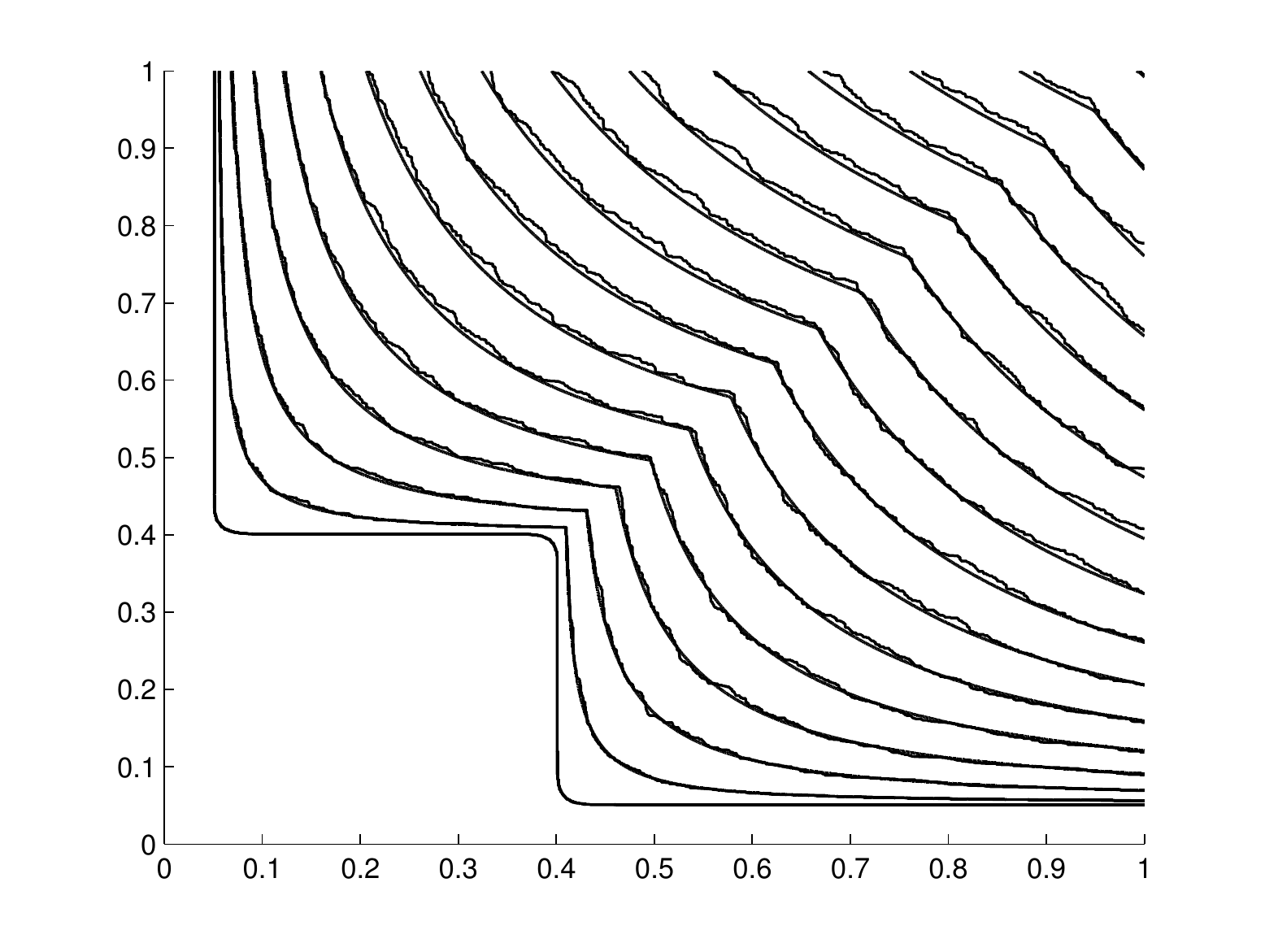}}
\caption{Comparison of the Pareto fronts and the level sets of $U$, where $U_xU_y = f$ and $f$ is the uniform density on the gray region in the top left.  The plots correspond to the Pareto fronts computed with $n=10^4$, $n=10^5$ and $n=10^6$ independent samples from $f$.  In each case, we show 15 equally spaced Pareto fronts and the corresponding level sets of $U$.}
\label{fig:uniform}
\end{figure}

For $d=2$, we have $c_2=2$, hence the level sets of $U$ will approximate the Pareto fronts, where $U_xU_y = f$.  We now show examples of Pareto fronts alongside the level sets of $U$ for $X_1,\dots,X_n$ sampled according to different density functions.  
In Figure \ref{fig:uniform}, we consider a uniform density on a portion of the unit square and show the Pareto fronts for $n=10^4$, $n=10^5$ and $n=10^6$ independent samples alongside the corresponding level sets of $U$.  Observing the Figure, we see that the Pareto fronts are well approximated by the level sets of $U$ for large $n$.  We also notice that the level sets of $U$ appear to yield a consistent underestimate of the Pareto fronts.  
Bollob\'as and Brightwell~\cite{bollobas1992} showed that the normalized expectation of the longest increasing subsequence among $n$ points chosen independently from the uniform distribution on $[0,1]^d$ is always bounded above by $c_d$, which is the limit of these normalized expectations as $n\to \infty$.  In light of this result, our observation is not surprising and merely confirms the results in~\cite{bollobas1992}.  We also observe that although the boundary of $\Omega$ is smooth, the solution $U$ develops shocks, or kinks, which are visible in the level sets of $U$.
\begin{figure}[t!]
\centering
\subfigure{\includegraphics[clip = true, trim = 30 25 25 20, height=0.27\textheight,width=0.475\textwidth]{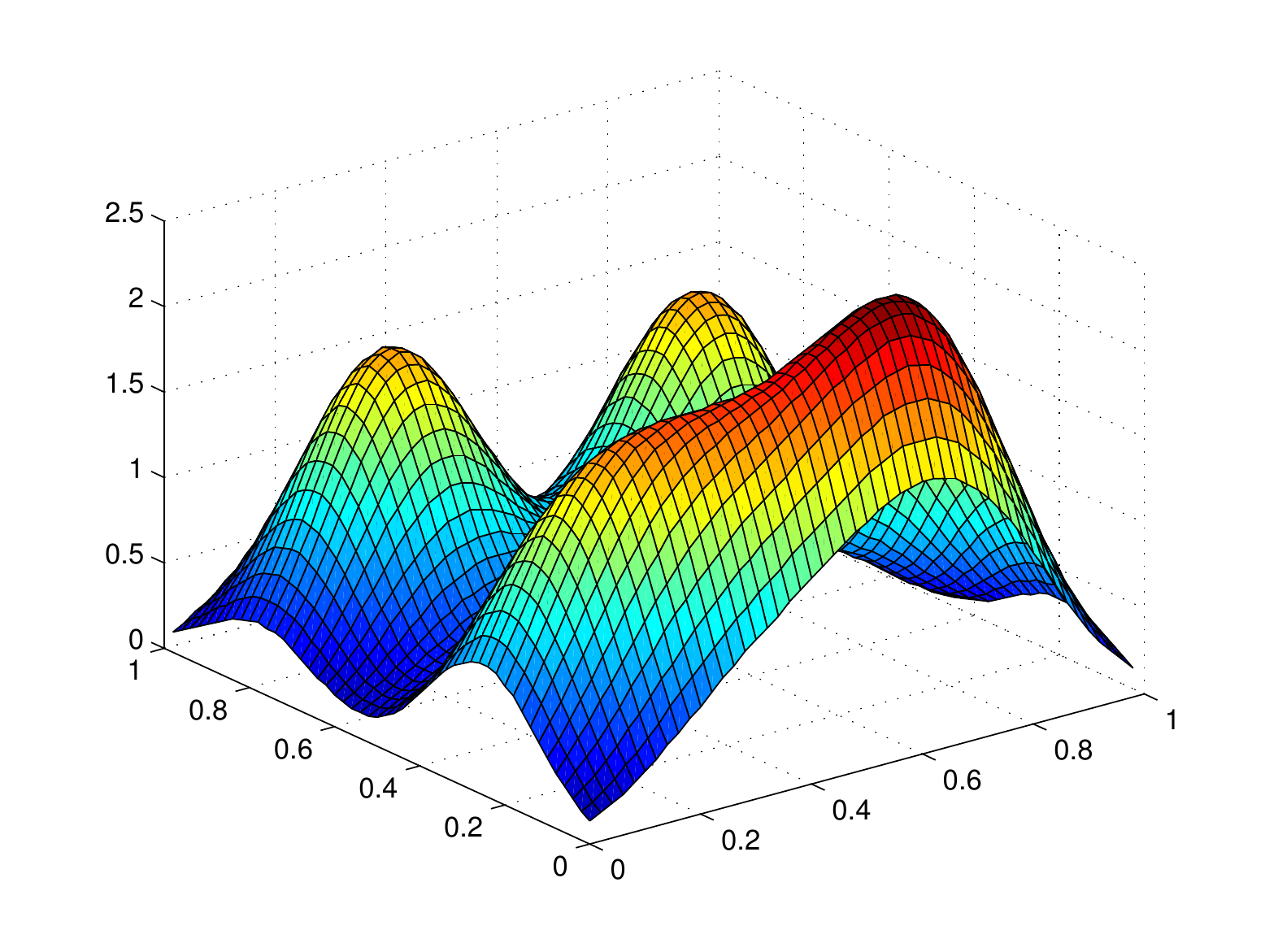}}
\subfigure{\includegraphics[clip = true, trim = 30 25 40 20, width=0.475\textwidth]{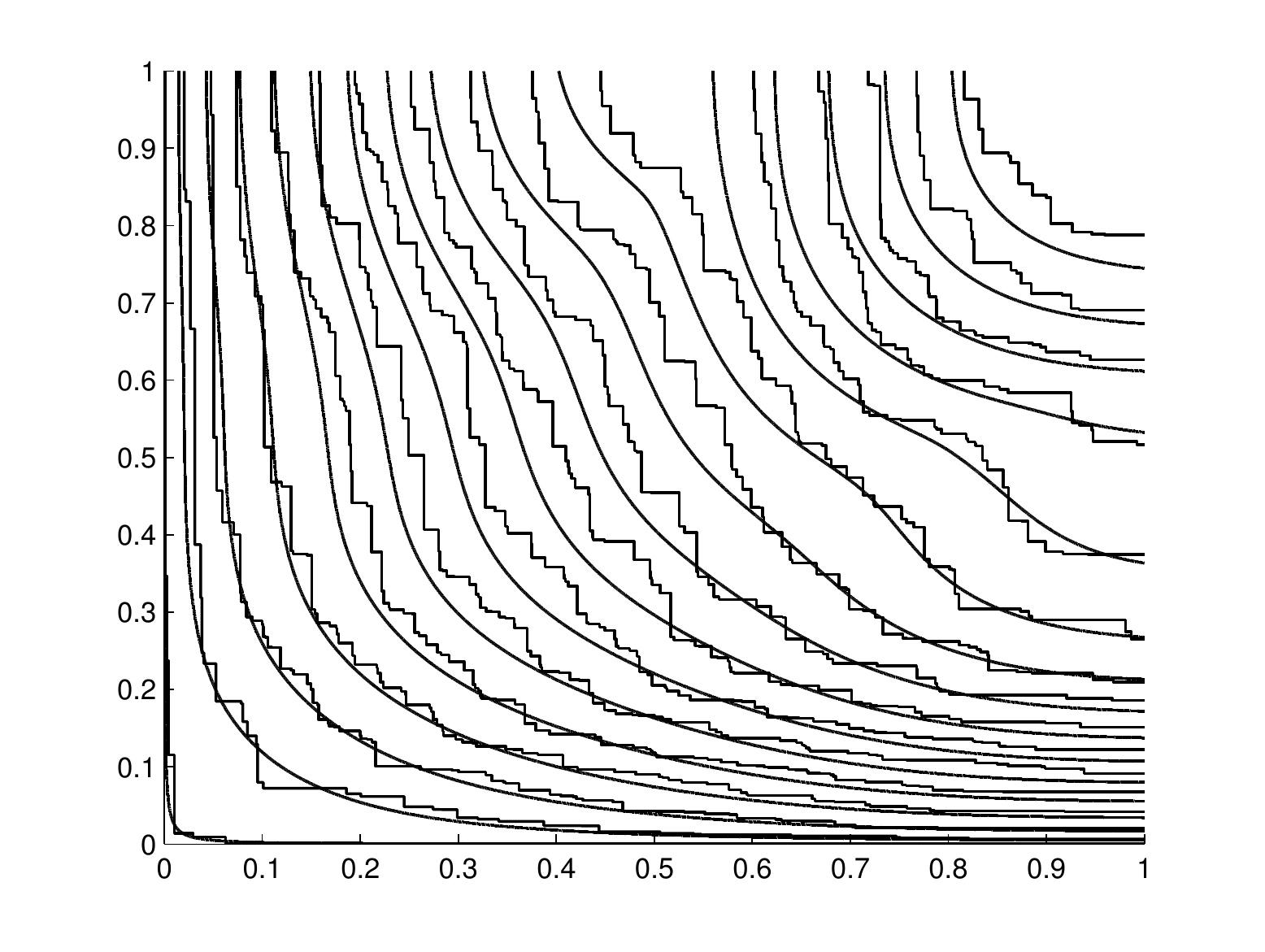}}\\
\subfigure{\includegraphics[clip = true, trim = 30 25 40 20, width=0.475\textwidth]{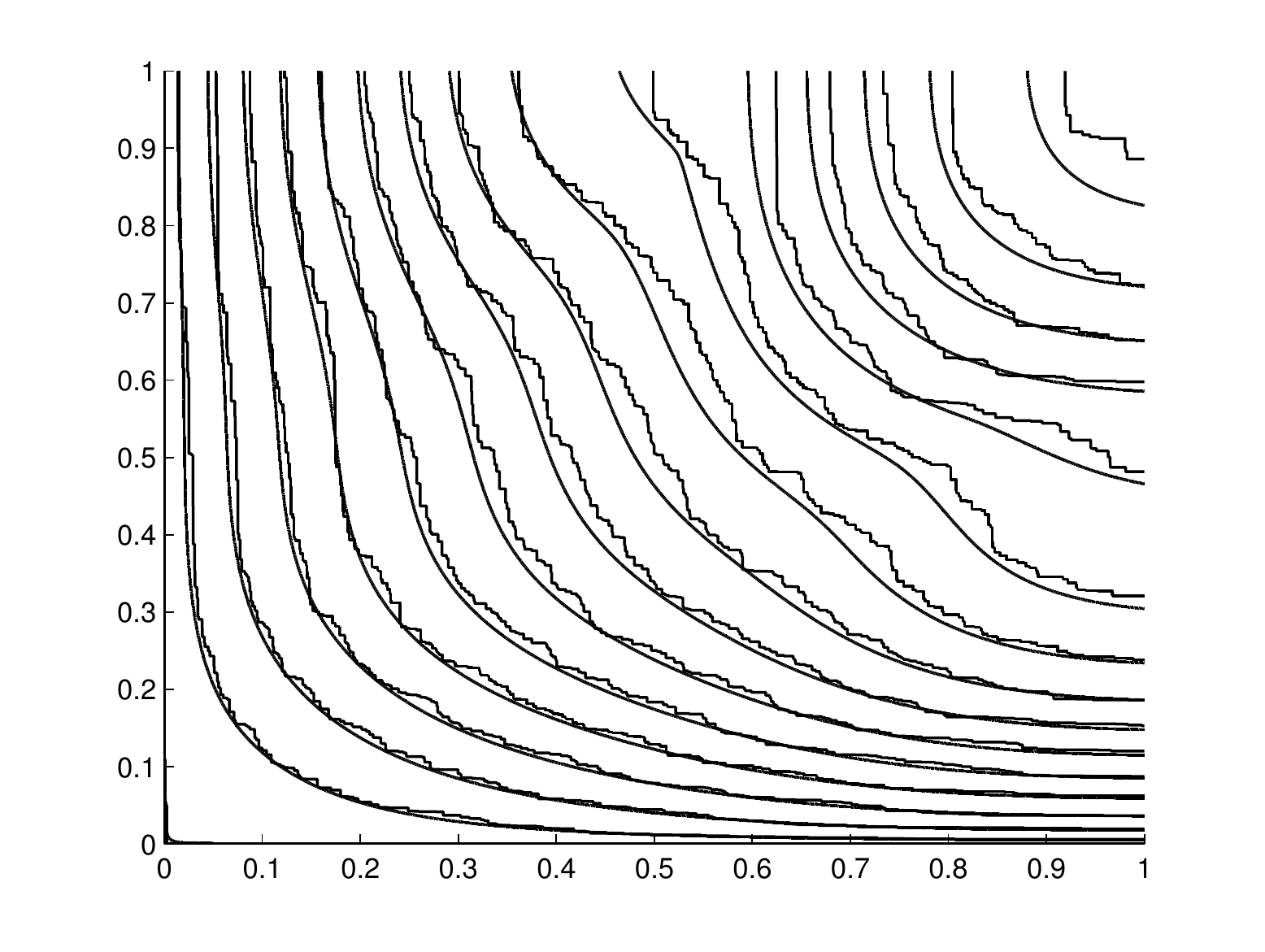}}
\subfigure{\includegraphics[clip = true, trim = 30 25 40 20, width=0.475\textwidth]{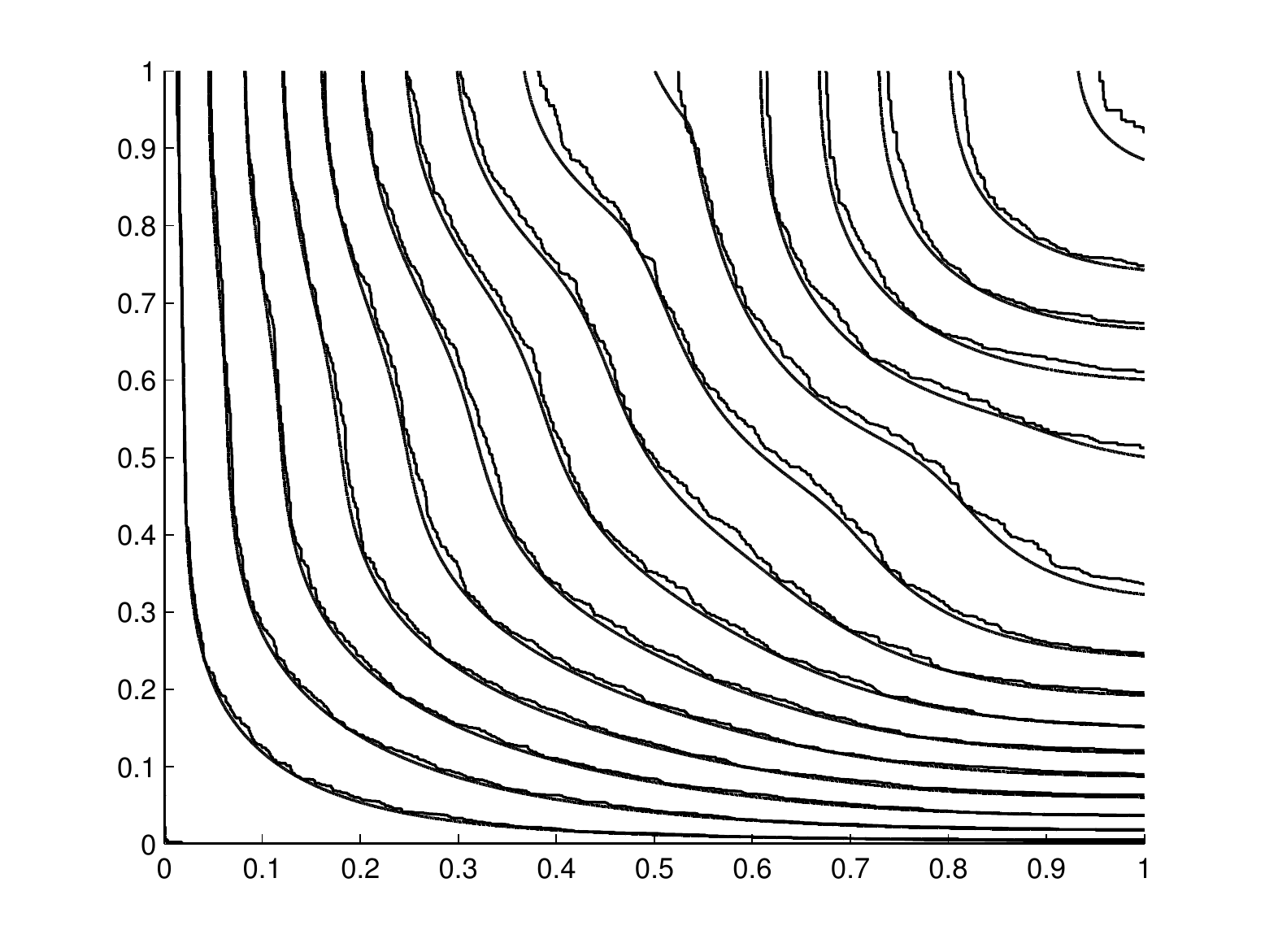}}
\caption{Comparison of the Pareto fronts and the level sets of $U$, where $U_x U_y = f$ and $f$ is the density depicted by the plot in the top left.  The plots correspond to the Pareto fronts computed with $n=10^4$, $n=10^5$ and $n=10^6$ independent samples from $f$. In each case, we show 15 equally spaced Pareto fronts and the corresponding level sets of $U$.}
\label{fig:face}
\end{figure}
In Figure \ref{fig:face}, we show the same comparison for a multi-modal density function on $[0,1]^2$.  The density function is depicted by the plot in Figure \ref{fig:face} and we have the same expected underestimation present here as well.

\section*{Acknowledgments}
We thank John Duchi for pointing out an error in the proof of Theorem \ref{thm:linfty-conv} in an earlier version of the manuscript

\end{document}